\documentclass[11pt, notitlepage]{article}   

\usepackage{amsmath,amsthm,amsfonts}   

\usepackage{amscd}
\usepackage{amsfonts}
\usepackage{amssymb}
\usepackage{color}      
\usepackage{epsfig}
\usepackage{graphicx}           

\usepackage{tikz}
\usepackage{tkz-graph}
\usepackage{tkz-berge}

\theoremstyle{plain}
\newtheorem{theorem}{Theorem}[section]
\newtheorem{lemma}[theorem]{Lemma}
\newtheorem{corollary}[theorem]{Corollary}
\newtheorem{proposition}[theorem]{Proposition}
\newtheorem{observation}[theorem]{Observation}
\newtheorem{remark}[theorem]{Remark}
\newtheorem{porism}[theorem]{Porism}

\theoremstyle{definition}

\newcommand{\cdim}{\textnormal{cdim}}
\newcommand{\diam}{\textnormal{diam}}
\newcommand{\rad}{\textnormal{rad}}
\newcommand{\rdiam}{\textnormal{rdiam}}
\newcommand{\rrad}{\textnormal{rrad}}
\newcommand{\code}{\textnormal{code}}

\def\finf{\mathop{{\rm I}\kern -.27 em {\rm F}}\nolimits}

\begin{document}

\date{}

\title{The connected metric dimension at a vertex of a graph}

\author{{\bf{Linda Eroh}}$^1$, {\bf{Cong X. Kang}}$^2$ and {\bf{Eunjeong Yi}}$^3$\\
\small University of Wisconsin Oshkosh, Oshkosh, WI 54901, USA$^1$\\
\small Texas A\&M University at Galveston, Galveston, TX 77553, USA$^{2,3}$\\
{\small\em eroh@uwosh.edu}$^1$; {\small\em kangc@tamug.edu}$^2$; {\small\em yie@tamug.edu}$^3$}

\maketitle

\begin{abstract}
The notion of metric dimension, $\dim(G)$, of a graph $G$, as well as a number of variants, is now well studied. In this paper, we begin a local analysis of this notion by introducing $\cdim_G(v)$, \emph{the connected metric dimension of $G$ at a vertex $v$}, which is defined as follows: a set of vertices $S$ of $G$ is a \emph{resolving set} if, for any pair of distinct vertices $x$ and $y$ of $G$, there is a vertex $z \in S$ such that the distance between $z$ and $x$ is distinct from the distance between $z$ and $y$ in $G$. We say that a resolving set $S$ is \emph{connected} if $S$ induces a connected subgraph of $G$. Then, $\cdim_G(v)$ is defined to be the minimum of the cardinalities of all connected resolving sets which contain the vertex $v$. The \emph{connected metric dimension of $G$}, denoted by $\cdim(G)$, is $\min\{\cdim_G(v): v \in V(G)\}$. Noting that $1 \le \dim(G) \le \cdim(G) \le \cdim_G(v) \le |V(G)|-1$ for any vertex $v$ of $G$, we show the existence of a pair $(G,v)$ such that $\cdim_G(v)$ takes all positive integer values from $\dim(G)$ to $|V (G)|-1$, as $v$ varies in a fixed graph $G$. We characterize graphs $G$ and their vertices $v$ satisfying $\cdim_G(v) \in \{1, |V(G)|-1\}$. We show that $\cdim(G)=2$ implies $G$ is planar, whereas it is well known that there is a non-planar graph $H$ with $\dim(H)=2$. We also characterize trees and unicyclic graphs $G$ satisfying $\cdim(G)=\dim(G)$. We show that $\cdim(G)-\dim(G)$ can be arbitrarily large. We determine $\cdim(G)$ and $\cdim_G(v)$ for some classes of graphs. We further examine the effect of vertex or edge deletion on the connected metric dimension. We conclude with some open problems.
\end{abstract}



\section{Introduction}

Let $G$ be a finite, simple, undirected, and connected graph of order at least two with vertex set $V(G)$ and edge set $E(G)$. The \emph{distance} between two vertices $x, y \in V(G)$, denoted by $d_G(x, y)$, is the length of a shortest path between $x$ and $y$ in $G$, and the distance between a vertex $v \in V(G)$ and a set $S \subseteq V(G)$, denoted by $d_G(v, S)$, is $\min\{d_G(v,x): x \in S\}$; we drop the subscript $G$ when no ambiguity arises. The \emph{eccentricity}, $e(v)$, of a vertex $v \in V(G)$ is $\max\{d(v,x) : x \in V(G)\}$. The \emph{radius}, $\rad(G)$, of $G$ is $\min\{ e(v): v \in V(G)\}$ and the \emph{diameter}, $\diam(G)$, of $G$ is $\max\{ e(v): v \in V(G)\}$; note that $\rad(G) \le \diam(G) \le 2 \rad(G)$. A vertex $u \in V(G)$ with $e(u)=\rad(G)$ is called a \emph{central vertex} of $G$. The \emph{open neighborhood} of a vertex $v \in V(G)$ is $N(v)=\{u \in V(G) \mid uv \in E(G)\}$, and its \emph{closed neighborhood} is $N[v]=N(v) \cup \{v\}$. The \emph{degree} of a vertex $u$ in $G$, denoted by $\deg_G(u)$, is $|N(u)|$; an end vertex is a vertex of degree one, and a \emph{major vertex} is a vertex of degree at least three. For a fixed graph $G$, an end vertex $\ell$ is called a \emph{terminal vertex} of a major vertex $v$ if $d(\ell, v)<d (\ell, w)$ for every other major vertex $w$ in $G$. The \emph{terminal degree}, $ter_G(v)$, of a major vertex $v$ is the number of terminal vertices of $v$ in $G$, and an \emph{exterior major vertex} is a major vertex that has positive terminal degree. We denote by $ex(G)$ the number of exterior major vertices of $G$, and $\sigma(G)$ the number of end vertices of $G$. A vertex $v\in V(G)$ is called a \emph{cut-vertex} if $G-v$ is disconnected. For $S \subseteq V(G)$, we denote by $G[S]$ the subgraph of $G$ induced by $S$. The \emph{join} of two graphs $H_1$ and $H_2$, denoted by $H_1+H_2$, is the graph obtained from the disjoint union of $H_1$ and $H_2$ by joining every vertex of $H_1$ with every vertex of $H_2$. The \emph{complement} of $G$, denoted by $\overline{G}$, has vertex set $V(\overline{G})=V(G)$ and edge set $E(\overline{G})$ such that $xy \in E(\overline{G})$ if and only if $xy \not\in E(G)$ for any distinct $x,y \in V(G)$. We denote by $P_n, C_n, K_n, K_{a,n-a}$, respectively, the path, the cycle, the complete graph, and the complete bi-partite graph on $n$ vertices.\\ 

A vertex $z \in V(G)$ \emph{resolves} a pair of vertices $x,y \in V(G)$ if $d(x,z) \neq d(y,z)$. For two distinct vertices $x,y \in V(G)$, let $R\{x,y\}=\{z \in V(G): d(x,z) \neq d(y,z)\}$. A set $S \subseteq V(G)$ is a \emph{resolving set} of $G$ if $|S\cap R\{x,y\}| \ge 1$ for every pair of distinct vertices $x$ and $y$ of $G$. The \emph{metric dimension} of $G$, denoted by $\dim(G)$, is the minimum cardinality of $S$ over all resolving sets of $G$. Two vertices $u, w \in V(G)$ are called \emph{twin vertices} if $N(u)-\{w\}=N(w)-\{u\}$; note that $S \cap \{u,w\} \neq \emptyset$ for any resolving set $S$ of $G$. For an ordered set $S=\{u_1, u_2, \ldots, u_k\} \subseteq V(G)$ of distinct vertices, the \emph{metric code} (or \emph{code}, for short) of $v \in V(G)$ with respect to $S$ is the $k$-vector $\code_S(v) =(d(v, u_1), d(v, u_2), \ldots, d(v, u_k))$. The concept of metric dimension, introduced independently by  Slater~\cite{slater} and by Harary and Melter~\cite{harary}, is now very well studied. One of the first motivations for its study is that of robot navigation in a network modeled by a graph (see~\cite{tree2}). The robot determines its location in the network by ``landmarks" or transmitters placed at sites (vertices of the graph); thus, metric dimension is the minimum number of transmitters required for the robot to know its position at all times in the network modeled by the graph. It was noted in~\cite{NP} that determining the metric dimension of a graph is an NP-hard problem.\\

Now, after giving the matter mooted by the forgoing scenario some more thought, it is easy to imagine that the more serious constraint is not the number of transmitters (markers) available to a robot, an intruder or a virus in a computer network, or a brave landing party onto an alien spaceship (taking inspiration from sci-fi movies), but the number of sites (vertices of the graph) that must be consecutively visited, infected, or secured, before the entire network is in some sense ``located". For example, take $K_{1,3}$ and label its vertices by $0,1,2,3$, where $0$ denotes the vertex of degree $3$. Let $G$ be the graph obtained from $K_{1,3}$ by replacing the edge $\{0,3\}$ by the path $P_{n-2}$. Now, a landing party which lands on vertex $1$ or $2$ would need to visit only three sites to locate every site of $G$. But, starting at the third end vertex, it would need to visit $n-1$ sites prior to being able to do the same. Of course, the key assumption here is that a landing party can only move from site A to site B in one step if A and B are adjacent.\\

Keeping the above scenario in mind, we introduce the connected metric dimension at a vertex of a graph and consequent notions. A resolving set $S$ of $G$ is called a \emph{connected resolving set} of $G$ if $G[S]$ is connected, and the \emph{connected metric dimension} of $G$, denoted by $\cdim(G)$, is the minimum cardinality of $S$ over all connected resolving sets of $G$. For $v \in V(G)$, we define \emph{the connected metric dimension at $v$  of $G$}, denoted by $\cdim_G(v)$, to be the minimum cardinality of a resolving set of $G$ which contains $v$ and induces a connected subgraph of $G$; then $\cdim(G)=\min_{v\in V(G)}\{\cdim_G(v)\}$ for any vertex $v\in V(G)$. As an analogue of eccentricity in the context of resolvability, we define $\rdiam(G)=\max_{u \in V(G)}\{\cdim_G(u)\}$ to be the \emph{resolving diameter} of $G$; then $\cdim(G)$ can also be called the \emph{resolving radius}, denoted by $\rrad(G)$, of $G$. It then follows that the \emph{resolving center} of $G$, denoted by $RC(G)$, is $\{v \in V(G) : \cdim_G(v)=\cdim(G)\}$, and that the \emph{resolving periphery} of $G$, denoted by $RP(G)$, is $\{v \in V(G) : \cdim_G(v) =\rdiam(G)\}$.\\

This paper is organized as follows. In section~2, we recall, for later use, some results on metric dimension of graphs. In section~3, we obtain some results on connected metric dimension for general graphs. For a connected graph $G$ and for $v\in V(G)$, we make useful observations involving $\cdim(G)$, $\cdim_G(v)$, and $\dim(G)$. We show that $\rrad(G) \le \rdiam(G) \le \rrad(G)+\diam(G)$ and give an example showing that $\rdiam(G)-\rrad(G)$ can be arbitrarily large. We further characterize graphs $G$ of order $n \ge 2$ satisfying $\cdim(G) \in \{1, n-1\}$ and vertices $v\in V(G)$ satisfying $\cdim_G(v) \in \{1,n-1\}$. In section~4, we show that $\cdim(G)=2$ implies $G$ is a planar graph; this is an interesting result considering the well-known fact that $\dim(G)=2$ does not imply the planarity of $G$ (see~\cite{tree2}). Of course, $\cdim(G)=2$ is simply saying that $G$ is resolved by $\{x,y\}$ where $xy \in E(G)$. We show an example of a non-planar graph $G$ resolved by $\{x,y\} \subseteq V(G)$ such that $d(x,y)=2$. For each positive integer $k \ge 3$, we construct a non-planar graph $G$ satisfying $\cdim(G)=k$. In section~5, we determine $\cdim(G)$ and $\cdim_G(v)$, for $v \in V(G)$, when $G$ is a tree, the Petersen graph, a wheel graph, a bouquet of $m$ cycles (the vertex sum of $m$ cycles at one common vertex) for $m \ge 2$, a complete multi-partite graph, or a grid graph (the Cartesian product of two paths). In section~6, we examine the effect of vertex or edge deletion on the connected metric dimension. In section~7, noting that $\cdim(G) \ge \dim(G)$ and that $\cdim(G)-\dim(G)$ can be arbitrarily large, we characterize graphs $G$ satisfying $\cdim(G)=\dim(G)$ when $G$ is a tree or a unicyclic graph. In section~8, we conclude this paper with some open problems.


\section{Preliminaries}

In this section, we recall for later use some known results on metric dimension.

\begin{theorem}\emph{\cite{tree1}}\label{dim_bounds}
For a connected graph $G$ of order $n \ge 2$ and diameter $d$,
$$f(n,d) \le \dim(G) \le n-d,$$
where $f(n,d)$ is the least positive integer $k$ for which $k+d^k \ge n$.
\end{theorem}

\begin{theorem}\emph{\cite{tree1}}\label{dim_characterization}
Let $G$ be a connected graph of order $n \ge 2$. Then
\begin{itemize}
\item[(a)] $\dim(G)=1$ if and only if $G=P_n$,
\item[(b)] $\dim(G)=n-1$ if and only if $G=K_n$,
\item[(c)] for $n \ge 4$, $\dim(G) =n-2$ if and only if $G=K_{s,t}$ ($s,t \ge 1$), $G=K_s+\overline{K_t}$ ($s \ge 1, t \ge 2$),or $G=K_s+(K_1 \cup K_t)$ ($s,t \ge 1$).
\end{itemize}
\end{theorem}

\begin{theorem}\emph{\cite{tree1, tree2, tree3}}\label{dim_tree}
For a tree $T$ that is not a path, $\dim(T)=\sigma(T)-ex(T)$.
\end{theorem}

\begin{theorem}\emph{\cite{P_dim}}\label{dim_Petersen}
For the Petersen graph $\mathcal{P}$, $\dim(\mathcal{P})=3$.
\end{theorem}

\begin{theorem}\emph{\cite{wheel1, wheel2}}\label{dim_wheel}
For $n \ge 3$, let $W_{n}=C_n+K_1$ be the wheel graph on $n+1$ vertices. Then
\begin{equation*}
\dim(W_{n})=\left\{
\begin{array}{cl}
3 & \mbox{ if } n\in\{3,6\},\\
\lfloor \frac{2n+2}{5}\rfloor & \mbox{ otherwise.}
\end{array}
\right.
\end{equation*}
\end{theorem}

\begin{theorem}\emph{\cite{bouquet}}\label{dim_bouquet}
For $m \ge 2$, let $B_m$ be a bouquet of $m$ cycles with a cut-vertex (i.e., the vertex sum of $m$ cycles at one common vertex). If $x$ is the number of even cycles of $B_m$, then
\begin{equation*}
\dim(B_m)=\left\{
\begin{array}{ll}
m & \mbox{ if } x=0,\\
m+x-1 & \mbox{ if } x \ge 1.
\end{array}\right.
\end{equation*}
\end{theorem}

\begin{theorem}\emph{\cite{K_dim}}\label{dim_multipartite}
For $k \ge 2$, let $G=K_{a_1, a_2, \ldots, a_k}$ be a complete $k$-partite graph of order $n=\sum_{i=1}^{k}a_i$. Let $s$ be the number of partite sets of $G$ consisting of exactly one element. Then
\begin{equation*}
\dim(G)=\left\{
\begin{array}{ll}
n-k & \mbox{ if } s=0,\\
n+s-k-1 & \mbox{ if } s \neq 0.
\end{array}
\right.
\end{equation*}
\end{theorem}

\begin{proposition}\emph{\cite{Cartesian_dim}}\label{dim_grid}
For the grid graph $G=P_s \square P_t$ ($s,t \ge 2$), $\dim(G)=2$.
\end{proposition}


\section{Some general results on connected metric dimension}

In this section, we record some useful observations involving the connected metric dimension of graphs. We observe the relation between the resolving diameter and the resolving radius of a graph. We show the existence of a pair $(G,v)$ such that $\cdim_G(v)$ takes all positive integer values from $\dim(G)$ to $|V(G)|-1$, as $v$ varies in a fixed graph $G$. We characterize pairs $(G, v)$ for which $\cdim_G(v)=1$ and $\cdim_G(v)=|V(G)|-1$, respectively. We begin with some useful observations for general graphs. We first recall the following well-known result.

\begin{lemma}\label{cutvertex}
Every connected graph of order at least two contains at least two vertices that are not cut-vertices.
\end{lemma}

\begin{observation}\label{obs}
Let $G$ be a connected graph of order $n \ge 2$. Then
\begin{itemize}
\item[(a)] $\dim(G) \le \cdim(G) \le n-1$, and thus $\cdim(G)=1$ if and only if $G=P_n$, by part (a) of Theorem~\ref{dim_characterization};
\item[(b)] $\cdim(G) \le \cdim_G(v) \le n-1$ for any $v \in V(G)$;
\item[(c)] $\cdim(G)=\dim(G)$ if and only if there exists a minimum resolving set $S$ of $G$ such that $S$ induces a connected subgraph of $G$;
\item[(d)] if $xy \in E(G)$, then $|\cdim_G(x)-\cdim_G(y)| \le 1$.
\end{itemize}
\end{observation}

\begin{proof}
Let $G$ be a connected graph of order $n \ge 2$.\\

(a) Since a minimum connected resolving set of $G$ is a resolving set of $G$, $\dim(G) \le \cdim(G)$. By Lemma~\ref{cutvertex}, there exists $u\in V(G)$ such that $u$ is not a cut-vertex of $G$; then $S=V(G)-\{u\}$ forms a connected resolving set of $G$; thus $\cdim(G) \le n-1$.\\

(b) For any $v \in V(G)$, $\cdim_G(v) \ge \min_{u \in V(G)}\{\cdim_G(u)\}=\cdim(G)$. Next, let $v \in V(G)$. By Lemma~\ref{cutvertex}, there exist two distinct vertices, say $u$ and $w$, in $G$ that are not cut-vertices. Assume, without loss of generality, that $u\neq v$, then $S=V(G)-\{u\}$ is a connected resolving set at $v$ and thus $\cdim_G(v)\leq n-1$.\\

(c) This is immediate from the two definitions. \\

(d) Put the inequality as $-1\leq \cdim_G(x)-\cdim_G(y) \leq 1$, and let $xy \in E(G)$. If $S$ is a minimum connected resolving set of $G$ with $x \in S$, then $S \cup \{y\}$ is a connected resolving set of $G$ containing $y$; thus $\cdim_G(y) \le |S|+1=\cdim_G(x)+1$. Swapping the roles played by $x$ and $y$ yields the other inequality.~\hfill
\end{proof}

\begin{observation}\label{obs_transitive}
Let $G$ be a connected, vertex-transitive graph. Then, for any $v \in V(G)$, $$\cdim_G(v)=\cdim(G).$$
\end{observation}

\begin{observation}\label{r=dim}
Let $G$ be a connected graph, and let $v \in V(G)$. Let $\mathfrak{C}(G)$ be the collection of all connected resolving sets $S$ of $G$ with $|S|=\dim(G)$. Then $\cdim_G(v)=\dim(G)$ if and only if $\mathfrak{C}(G)\neq\emptyset$ and $v \in S'$ for some $S'\in \mathfrak{C}(G)$.
\end{observation}

Next, we show the existence of a graph $G$ and its vertices such that $\cdim_G(v)$ takes all positive integer values on the closed interval $[\dim(G), |V(G)|-1]$ as $v$ varies in $G$.

\begin{proposition}\label{Gab}
There is a graph $G$ such that, for any integer $m$ with $\dim(G) \le m \le |V(G)|-1$, there exists a vertex $v\in V(G)$ satisfying $\cdim_G(v)=m$.
\end{proposition}

\begin{proof}
For $a \ge 3$ and $b \ge 1$, let $G=G_{a,b}$ be the graph obtained from the disjoint union of $K_a$ and $P_b$ by joining an edge between a vertex of $K_a$ and an end vertex of $P_b$ (see Figure~\ref{realization} for $G_{6,5}$). Let $V(K_a)=\{u_0, u_1, u_2, \ldots, u_{a-1}\}$ and let $P_b$ be given by $w_1, w_2, \ldots, w_b$ such that $u_0w_1\in E(G)$; note that $|V(G)|=a+b$.

First, we show that $\dim(G)=a-1$. Let $S$ be any minimum resolving set of $G$. Since any two vertices in $\cup_{i=1}^{a-1}\{u_i\}$ are twin vertices, $|S \cap (\cup_{i=1}^{a-1}\{u_i\})| \ge a-2$. Let $S^*=S \cap (\cup_{i=1}^{a-1}\{u_i\})$. If $|S^*|=a-2$, we can assume, without loss of generality, $S^*=\cup_{i=1}^{a-2}\{u_i\}$. But then $\code_{S^*}(u_0)=\code_{S^*}(u_{a-1})$; thus $\dim(G) \ge a-1$. Since $S=\cup_{i=0}^{a-2}\{u_i\}$ forms a resolving set of $G$ with $|S|=a-1$, we have $\dim(G)=a-1$. 

Second, we show that $\cdim_{G}(v)$ takes all integer values on $[a-1, a+b-1]$ for some $v \in V(G)$; notice that $a-1=\dim(G)$ and $a+b-1=|V(G)|-1$. With $S$ as defined, we have $\cdim_{G}(u_0)=a-1$ by Observation~\ref{r=dim}. Further, we have $\cdim_{G}(w_j)=a+j-1$ for $j \in \{1,2,\ldots, b\}$, since any minimum connected resolving set $S'$ at $w_j$ must satisfy $|S' \cap(\cup_{i=0}^{a-1}\{u_i\})|=a-1$ and $S' \supseteq\{w_1, w_2, \ldots, w_{j-1}, w_j\}$.~\hfill
\end{proof}

\begin{figure}[ht]
\centering
\begin{tikzpicture}[scale=.5, transform shape]

\node [draw, shape=circle, scale=1.3] (1) at  (0,3) {};
\node [draw, shape=circle, scale=1.3] (2) at  (-2.6, 1.5) {};
\node [draw, shape=circle, scale=1.3] (3) at  (-2.6, -1.5) {};
\node [draw, shape=circle, scale=1.3] (4) at  (0, -3) {};
\node [draw, shape=circle, scale=1.3] (5) at  (2.6, -1.5) {};
\node [draw, shape=circle, scale=1.3] (0) at  (2.6, 1.5) {};

\node [draw, shape=circle, scale=1.3] (11) at  (4.6, 1.5) {};
\node [draw, shape=circle, scale=1.3] (22) at  (6.6, 1.5) {};
\node [draw, shape=circle, scale=1.3] (33) at  (8.6, 1.5) {};
\node [draw, shape=circle, scale=1.3] (44) at  (10.6, 1.5) {};
\node [draw, shape=circle, scale=1.3] (55) at  (12.6, 1.5) {};

\node [scale=1.9] at (0,3.55) {$u_1$};
\node [scale=1.9] at (-3.3,1.5) {$u_2$};
\node [scale=1.9] at (-3.3,-1.5) {$u_3$};
\node [scale=1.9] at (0,-3.55) {$u_4$};
\node [scale=1.9] at (3.3,-1.5) {$u_5$};
\node [scale=1.9] at (2.75,2) {$u_0$};

\node [scale=1.9] at (4.6,2) {$w_1$};
\node [scale=1.9] at (6.6,2) {$w_2$};
\node [scale=1.9] at (8.6,2) {$w_3$};
\node [scale=1.9] at (10.6,2) {$w_4$};
\node [scale=1.9] at (12.6,2) {$w_5$};

\draw(1)--(2)--(3)--(4)--(5)--(0)--(1);\draw(1)--(3)--(5)--(1); \draw(2)--(4)--(0)--(2);
\draw(1)--(4); \draw(2)--(5); \draw(3)--(0); \draw(0)--(11)--(22)--(33)--(44)--(55);

\end{tikzpicture}
\caption{$G_{6,5}$}\label{realization}
\end{figure}

Next, we consider the relation between $\rrad(G)$ and $\rdiam(G)$. For the usual eccentricity, $\diam(G) \le 2 \rad(G)$; however, a similar relationship between resolving diameter and resolving radius does not hold.

\begin{proposition}\label{rrad_rdiam}
For a connected graph $G$ of order at least two, $$\rrad(G) \le \rdiam(G) \le \rrad(G)+\diam(G),$$
where both bounds are sharp.
\end{proposition}

\begin{proof}
The lower bound is Observation~\ref{obs}(b) and follows directly from the definitions. For the upper bound, let $x,y \in V(G)$ such that $\cdim_G(x)=\rdiam(G)$, $\cdim_G(y)=\rrad(G)$, and $S$ be a minimum connected resolving set at $y$. Let $y=y_0, y_1, \ldots, y_{k-1}, y_k=x$ denote the vertices along a shortest path from $y$ to $x$. Then $S\cup\{y_1,\ldots,y_{k-1},y_k=x\}$ is a connected resolving set at $x$, and $k\leq \diam(G)$ by definition of the diameter. 

For the sharpness of the lower bound, let $G$ be any vertex-transitive graph; then $\rdiam(G)=\rrad(G)$ by Observation~\ref{obs_transitive}.

For the sharpness of the upper bound, let $G$ be the graph obtained from $P_{n-4}$, for $n \ge 6$, by joining an end vertex of $P_{n-4}$ to every vertex of $P_4$ (see Figure~\ref{difficult}). It is easy to check that $\{u_2, u_3\}$ forms a connected resolving set of $G$, and hence $\rrad(G)=\cdim(G)=2$ by Observation~\ref{obs}(a). Notice that $\diam(G)=n-4$, since $u_i,w_1,\ldots,w_{n-4}$ yield a diametral path for $1\leq i\leq 4$. It is also easy to see that $\cdim_G(w_{n-4})=n-2=\rdiam(G)$, since $w_1$ can not resolve $u_i$ from $u_j$.~\hfill
\end{proof}

\begin{figure}[ht]
\centering
\begin{tikzpicture}[scale=.5, transform shape]

\node [draw, shape=circle, scale=1.1] (1) at  (0,3) {};
\node [draw, shape=circle, scale=1.1] (2) at  (0,1) {};
\node [draw, shape=circle, scale=1.1] (3) at  (0,-1) {};
\node [draw, shape=circle, scale=1.1] (4) at  (0,-3) {};

\node [draw, shape=circle, scale=1.1] (a) at  (2,0) {};
\node [draw, shape=circle, scale=1.1] (b) at  (4,0) {};
\node [draw, shape=circle, scale=1.1] (c) at  (6,0) {};
\node [draw, shape=circle, scale=1.1] (d) at  (9.5,0) {};

\node [scale=1.9] at (-2.5,0) {\large $G$:};

\node [scale=1.9] at (-0.6,3) {$u_1$};
\node [scale=1.9] at (-0.6,1) {$u_2$};
\node [scale=1.9] at (-0.6,-1) {$u_3$};
\node [scale=1.9] at (-0.6,-3) {$u_4$};

\node [scale=1.9] at (2.15,0.5) {$w_1$};
\node [scale=1.9] at (4.1,0.5) {$w_2$};
\node [scale=1.9] at (6.1,0.5) {$w_3$};
\node [scale=1.9] at (9.7,0.5) {$w_{n-4}$};

\draw(a)--(1)--(2)--(3)--(4)--(a)--(b)--(c)--(7,0);\draw(2)--(a)--(3);\draw[thick, dotted](7,0)--(8.5,0);\draw(8.5,0)--(d);

\end{tikzpicture}
\caption{A graph $G$ satisfying $\rdiam(G) = \rrad(G)+\diam(G)$.}\label{difficult}
\end{figure}
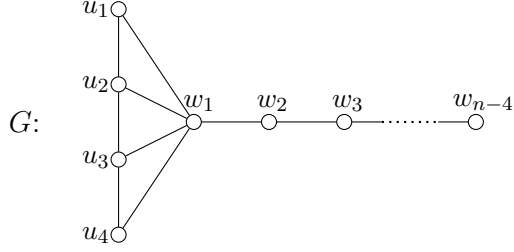

\begin{remark}
For a fixed graph $G$, $\rdiam(G)-\rrad(G)$ can be arbitrarily large. Let $G=G_{a,b}$ as described in the proof of Proposition~\ref{Gab}. Then $\rdiam(G)-\rrad(G)=\cdim_G(w_b)-\cdim_G(u_0)=a+b-1-(a-1)=b$ can be arbitrarily large as $b \rightarrow \infty$.
\end{remark}

It is a result of Harary and Norman~\cite{HN} that all central vertices of any connected graph $G$ are contained in a subgraph of $G$ with no cut-vertices. Figure~\ref{resolvingcen} shows that a similar result does not hold for resolving centers.

\begin{remark}
\begin{itemize}
\item[(a)] There exists a graph $G$ such that the subgraph induced by $RC(G)$ is not connected. Let $G$ be the graph in Figure \ref{resolvingcen}, and we label each vertex $v\in V(G)$ by $\cdim_G(v)$; then $G[RC(G)]=2K_{1,3}$.
\item[(b)] There exists a graph $G$ such that any connected subgraph of $G$ containing $RC(G)$ contains cut-vertices (see Figure \ref{resolvingcen}).
\item[(c)] Note that $RC(G)$ is the union of all minimum connected resolving sets of $G$.
\item[(d)] Let $H_1, H_2, \ldots, H_k$ be the $k$ connected components of $G[RC(G)]$, where $k \ge 1$. Then, for each $i\in\{1,2,\ldots, k\}$, $V(H_i)$ forms a (not necessarily minimum) resolving set of $G$.
\end{itemize}
\end{remark}

\begin{figure}[ht]
\centering
\begin{tikzpicture}[scale=.5, transform shape]

\node [draw, shape=circle, scale=1.1] (a) at  (0,0) {};
\node [draw, shape=circle, scale=1.1] (b) at  (1.5,0) {};
\node [draw, shape=circle, scale=1.1] (c) at  (9,0) {};
\node [draw, shape=circle, scale=1.1] (d) at  (10.5,0) {};

\node [draw, shape=circle, scale=1.1] (1) at  (3,1) {};
\node [draw, shape=circle, scale=1.1] (2) at  (4.5,1) {};
\node [draw, shape=circle, scale=1.1] (3) at  (6,1) {};
\node [draw, shape=circle, scale=1.1] (4) at  (7.5,1) {};
\node [draw, shape=circle, scale=1.1] (11) at  (3,-1) {};
\node [draw, shape=circle, scale=1.1] (22) at  (4.5,-1) {};
\node [draw, shape=circle, scale=1.1] (33) at  (6,-1) {};
\node [draw, shape=circle, scale=1.1] (44) at  (7.5,-1) {};

\node [scale=1.9] at (-2.5,0) {\large $G$:};

\node [scale=1.9] at (0,0.6) {$3$};
\node [scale=1.9] at (1.5,0.6) {$3$};
\node [scale=1.9] at (9,0.6) {$3$};
\node [scale=1.9] at (10.5,0.6) {$3$};

\node [scale=1.9] at (3,1.6) {$3$};
\node [scale=1.9] at (4.5,1.6) {$4$};
\node [scale=1.9] at (6,1.6) {$4$};
\node [scale=1.9] at (7.5,1.6) {$3$};
\node [scale=1.9] at (3,-1.6) {$3$};
\node [scale=1.9] at (4.5,-1.6) {$4$};
\node [scale=1.9] at (6,-1.6) {$4$};
\node [scale=1.9] at (7.5,-1.6) {$3$};

\draw(a)--(b)--(1)--(2)--(3)--(4)--(c)--(d);\draw(b)--(11)--(22)--(33)--(44)--(c);

\end{tikzpicture}
\caption{A connected graph $G$ such that $G[RC(G)]$ is not connected and that any connected subgraph of $G$ containing $RC(G)$ contains cut-vertices; here, each vertex $v$ is labeled by $\cdim_G(v)$.}\label{resolvingcen}
\end{figure}
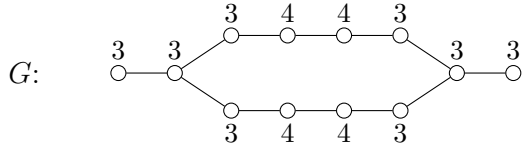

Next, we characterize graphs (vertices in graphs, respectively) with $\cdim(G)=1$ or $\cdim(G)=|V(G)|-1$ ($\cdim_G(v)=1$ or $\cdim_G(v)=|V(G)|-1$, respectively). For $2 \le r \le n-2$, let $P_r$ be the path given by $u_1, u_2, \ldots, u_r$. Let $\mathcal{H}_1$ be the family of graphs (forks) obtained from a disjoint union of $P_r$ and $\overline{K}_{n-r}$ by joining $u_r$ to every vertex of $\overline{K}_{n-r}$. Let $\mathcal{H}_2$ be the family of graphs (paddles) obtained from a disjoint union of $P_r$ and $K_{n-r}$ by joining $u_r$ to every vertex of $K_{n-r}$; note that $\mathcal{H}_2$ corresponds to the family of graphs $G_{n-r+1,r-1}$ described in the proof of Proposition~\ref{Gab}.

\begin{theorem}\label{cdim_characterization}
Let $G$ be a connected graph of order $n \ge 2$, and let $v\in V(G)$. Then
\begin{itemize}
\item[(a)] $\cdim_G(v)=1$ if and only if $G=P_n$ and $v$ is an end vertex of $G$,
\item[(b)] $\cdim_G(v)=n-1$ if and only if (i) $G \in \{K_2, K_3\} \cup \{K_n, K_{1,n-1}\}$, for $n \ge 4$, and $v$ is any vertex of $G$, or (ii) $G=K_{1,2}=P_3$ and $v$ is the degree-two vertex of $G$, or (iii) $G\in \mathcal{H}_1 \cup \mathcal{H}_2$ and $v=u_1$.
\end{itemize}
\end{theorem}

\begin{proof}
Let $G$ be a connected graph of order $n \ge 2$, and let $v\in V(G)$.\\

(a) ($\Leftarrow$) If $v$ is an end vertex of $P_n$, then $\{v\}$ is a resolving set of $P_n$; thus, $\cdim_G(v)=1$.

($\Rightarrow$) Let $\cdim_G(v)=1$. Then $G=P_n$ by Theorem~\ref{dim_characterization}(a) and Observation~\ref{obs}(a)(b). If $v$ is not an end vertex of $G=P_n$, then $v$ fails to resolve two vertices in $N(v)$; thus $\cdim_G(v) \ge 2$. So, $v$ must be an end vertex of $P_n$.\\

(b) ($\Leftarrow$) \textbf{Case 1: $G=K_n$ for $n \ge 2$.} In this case, $\cdim(G)=n-1$ by Theorem~\ref{dim_characterization}(b) and Observation~\ref{obs}(a). By Observation~\ref{obs}(b), $\cdim_G(v)=n-1$ for any $v\in V(G)$.

\textbf{Case 2: $G=K_{1,n-1}$ for $n \ge 3$.} If $n=3$ (i.e., $G=K_{1,2}=P_3$) and $v$ is the degree-two vertex of $G$, the claim is obvious. So, let $n \ge 4$. Let $w$ be the central vertex of $G$ and let $L=\{\ell_1, \ell_2, \ldots, \ell_{n-1}\}$ be the set of end vertices of $G$. Let $S$ be any minimum resolving set of $G$. Then $|S \cap L| =n-2=\dim(G)$ by Theorem~\ref{dim_tree} and the fact that any two distinct vertices in $L$ are twin vertices. So, $\cdim(G) = n-1$: (i) $\cdim(G) \ge |S|+1=n-1$ since $G[S]$ is disconnected; (ii) $\cdim(G) \le n-1$ by Observation~\ref{obs}(a). Thus, $\cdim_G(v)=n-1$, for any $v\in V(G)$, by Observation~\ref{obs}(b).

\textbf{Case 3: $G\in \mathcal{H}_1 \cup \mathcal{H}_2$ and $v=u_1$.} Let $S$ be any connected resolving set of $G$ with $u_1\in S$. If $G \in \mathcal{H}_1$, then $|S \cap V(\overline{K}_{n-r})| \ge n-r-1 \ge 1$ since any two distinct vertices in $\overline{K}_{n-r}$ are twin vertices of $G$ and $2 \le r \le n-2$; similarly, if $G \in \mathcal{H}_2$, then $|S \cap V(K_{n-r})| \ge n-r-1 \ge 1$ since any two distinct vertices in $K_{n-r}$ are twin vertices of $G$. Since $u_1 \in S$ and $G[S]$ is connected, $\cup_{i=1}^{r}\{u_i\} \subset S$; thus $|S| \ge r+(n-r-1)=n-1$, and hence $\cdim_G(u_1) \ge n-1$. By Observation~\ref{obs}(b), $\cdim_G(u_1)=n-1$.

($\Rightarrow$) Let $\cdim_G(v)=n-1$.

\textbf{Case 1: $e(v)=1$.} Note that $\deg_G(v)=n-1$. Suppose that $G \not\in \{K_2, K_3\} \cup \{K_n, K_{1,n-1}\}$; then $n\geq 3$, as $G$ is connected. If $n=3$, then $G=P_3$ and claim (ii) of (b) follows. So let $n \geq 4$; then, excluding $v$, there exist three distinct vertices $x,y,z$ such that $xy \in E(G)$ and $xz\not\in E(G)$. Since $N[v]=V(G)$ and $d(x,y)=1<d(x,z)$, it follows that $V(G)-\{y,z\}$ is a connected resolving set of $G$ at $v$, contradicting the assumption that $\cdim_G(v)=n-1$. (Thus, if $e(v)=1$, then (i) or (ii) of (b) holds.)

\textbf{Case 2: $e(v) =k \ge 2$.} Let $t \in V(G)$ be at distance $k$ from $v$. We claim that, for every positive integer $p \in [1, k-1]$, there is exactly one vertex $u\in V(G)$ with $d(v,u)=p$. Assume, to the contrary, that there exists $p\in[1,k-1]$ such that there are at least two distinct vertices at distance $p$ from $v$ in $G$, and take $q$ to be the maximum of all such $p$'s. Let $A=\{x_1,x_2, \ldots, x_{\alpha}\}$ be all vertices in $G$ at distance $q$ from $v$.

First, assume $q\leq k-2$; note that $k \ge 3$ in this case. Then there is exactly one vertex $y \in V(G)$, lying on a shortest $v-t$ path, such that $d(v,y)=q+1$. Without loss of generality, let $x_1y\in E(G)$. Then $V(G)-\{x_2,t\}$ forms a connected resolving set of $G$: (i) since $d(v,t)=e(v)$, no shortest $v-s$ path contains $t$ for $s \in V(G)-\{t\}$; (ii) for any vertex $s\in V(G)-\{x_2\}$ with $1\le d(v,s) \le q$, no shortest $v-s$ path contains $x_2$; (iii) since $y$ is the only vertex at distance $q+1$ from $v$ in $G$ and $x_1y\in E(G)$, for every $s \in V(G)$ with $q<d(v,s) \le k$, there is a shortest $v-s$ path that contains both $x_1$ and $y$, but not $x_2$; (iv) $v$ resolves $x_2$ from $t$ in $G$. Thus, $\cdim_G(v) \le n-2$, contradicting the assumption.

Now, let $q=k-1$. If $t$ is the only vertex at distance $q+1=k$ from $v$ such that $x_1$ lies on a shortest $v-t$ path, one can show, as before, that $V(G)-\{x_2, t\}$ is a connected resolving set of $G$, implying $\cdim_G(v) \le n-2$. So, assume $\beta\geq 2$ and let $B=\{z_1,z_2, \ldots, z_{\beta}\}$ be all vertices in $G$ at distance $q+1=k$ from $v$. Let $B'$ be a two-element subset of $B$. Since $V(G)-B'$ induces a connected subgraph of $G$, $\cdim_G(v)\leq n-2$ if there is an $x_i\in A$ which resolves the two vertices of $B'$. If no vertex in $A$ can resolve any pair of vertices in $B$, then $B\subseteq N(x_i)$ for each $x_i\in A$; in this case $V(G)-\{x_1,z_1\}$ yields a connected resolving set at $v$.

Thus, for every positive integer $p \in [1, k-1]$, there is exactly one vertex at distance $p$ from $v$, which implies that $v$ is an end vertex of $P_k \subset G$. Let the $P_k$ be given by $v=u_1, u_2, \ldots, u_k$, and let $C=\{w_1, w_2, \ldots, w_{n-k}\}$ be the set of vertices at distance $k$ from $v$; note that $n-k \ge 2$, otherwise, $G$ is a path and $\cdim_G(v)=1\le n-2$. If $G[C] \in \{\overline{K}_{n-k}, K_{n-k}\}$, then $G$ and $v$ satisfy (iii) of (b). If $G[C] \not\in \{\overline{K}_{n-k}, K_{n-k}\}$, then there are three distinct vertices $x$, $y$, and $z$ at distance $k$ from $v$ such that $xy \in  E(G)$ and $xz\not\in E(G)$. This again makes $V(G)-\{y,z\}$ a connected resolving set of $G$, contradicting $\cdim_G(v)=n-1$.~\hfill
\end{proof}

As an immediate consequence of Observation~\ref{obs}(b) and Theorem~\ref{cdim_characterization}, we have the following

\begin{corollary}
Let $G$ be a connected graph of order $n \ge 2$. Then $\cdim(G)=n-1$ if and only if $G \in \{K_2, K_3\} \cup \{K_n, K_{1,n-1}\}$, where $n \ge 4$.
\end{corollary}


\section{The planarity of $G$ is implied by $\cdim(G)=2$}

A graph is \emph{planar} if it can be drawn in a  plane without edge crossing. It is well known that $\dim(G)=2$ does not imply the planarity of $G$ (see~\cite{tree2}). In this section, we show that $\cdim(G)=2$ implies the planarity of $G$ by exhibiting, for any integer $n$ greater than $1$, all graphs of order $n$ having a fixed pair of adjacent vertices as a resolving set. For each positive integer $k \ge 3$, we show the existence of a non-planar graph $G$ with $\cdim(G)=k$. Given that $\cdim(G)\leq 2$ is the same as being able to resolve $G$ by two adjacent vertices, we show, for contrast, a non-planar graph resolved by two vertices at distance two from each other. \\

We first recall a result on the characterization of planar graphs due to Wagner~\cite{wagner}. For two graphs $G$ and $H$, $H$ is called a \emph{minor} of $G$ if $H$ can be obtained from $G$ by vertex deletion, edge deletion, or edge contraction. 

\begin{theorem}\label{planarity}\emph{\cite{wagner}}
A graph $G$ is planar if and only if neither $K_5$ nor $K_{3,3}$ is a minor of $G$.
\end{theorem}

For each nonnegative integer $r$, we define a family $\mathcal{F}_r$ of obviously planar graphs on a fixed, connected resolving set $\{u,w\}$ as follows. First, let $\mathcal{F}_0=\{P_2\}$, where $V(P_2)=\{u,w\}$. For $r\geq 1$, each graph $G$ of the family $\mathcal{F}_r$ must satisfy the following rules:
\begin{enumerate}
\item[(R1)] $\{u,w\}\subseteq V(G)\subseteq \{u,w\}\cup\{x_a,y_a,z_a:1\leq a\leq r\}$
\item[(R2)] if $x_1 \in V(G)$, then $x_1u \in E(G)$
\item[(R3)] if $z_1 \in V(G)$, then $z_1w \in E(G)$
\item[(R4)] if $y_1 \in V(G)$, then $y_1u, y_1w \in E(G)$
\item[(R5)] if $x_a \in V(G)$ for $2 \leq a \leq r$, then $x_{a-1} \in V(G)$ and $x_ax_{a-1} \in E(G)$
\item[(R6)] if $z_a \in V(G)$ for $2 \leq a \leq r$, then $z_{a-1} \in V(G)$ and $z_az_{a-1} \in E(G)$
\item[(R7)] if $y_a \in V(G)$, $2 \leq a \leq r$, then at least one of the following must be true:
\begin{enumerate}
\item[(i)] $x_{a-1}, z_{a-1} \in V(G)$ and $x_{a-1}y_a, z_{a-1}y_a \in E(G)$, or 
\item[(ii)] $y_{a-1} \in V(G)$ and $y_{a-1}y_a \in E(G)$
\end{enumerate}
\item[(R8)] any of the edges $x_ay_a$, $x_az_a$, or $y_az_a$ may or may not be in $E(G)$, for each $a$ with $1 \leq a \leq r$
\item[(R9)] there are no other edges in $G$.
\end{enumerate}

\begin{theorem}\label{thm_cdim2}  
For a connected graph $G$, $\cdim(G) \leq 2$ if and only if $G \in \mathcal{F}_r$ for a nonnegative integer $r$. 
\end{theorem}

\begin{proof}  ($\Leftarrow$) Suppose $G \in \mathcal{F}_r$ for a nonnegative integer $r$. Let $S=\{u,w\}$. The case of $r=0$ is trivial; so, let $r\geq 1$. It is easy to check, based on the rules imposed on graphs in $\mathcal{F}_r$, that $\code_S(x_a)=(a,a+1)$, $\code_S(y_a)=(a,a)$, and $\code_S(z_a)=(a+1,a)$, where $1\leq a\leq r$ . Thus, $S$ is a connected resolving set of $G$ and $\cdim(G) \leq 2$.\\

($\Rightarrow$) Suppose $\cdim(G) \leq 2$. Then there is a connected resolving set $S$ for $G$ satisfying $|S|=2$; put $S=\{u,w\}$.  Since $S$ is a connected resolving set, $uw \in E(G)$. Let $r = \max_{v \in V(G)}\{\min\{x,y\}: \code_S(v)=(x,y)\}$.  Notice, since $u$ and $w$ are adjacent, that the two entries in a code for a vertex of $G$ can differ by at most one.  Thus, the set of possible codes for vertices other than $u$ and $w$ are $\{(a,a+1),(a,a),(a+1,a) : 1 \leq a \leq r\}$, and rule (R1) follows. 

Let $x_a$ be the vertex with code $(a,a+1)$, if it exists, let $y_a$ be the vertex with code $(a,a)$, if it exists, and let $z_a$ be the vertex with code $(a+1,a)$, if it exists. Rules (R2), (R3), and (R4) follow, since a vertex can have code $(1,2)$ only if it is adjacent to $u$ and not $w$, code $(1,1)$ only if it is adjacent to both $u$ and $w$, and code $(2,1)$ only if it is adjacent to $w$ and not $u$.  

Regarding (R5), if a vertex $x_a$ has code $(a,a+1)$, then it must be adjacent to a vertex at distance $a-1$ from $u$ but cannot be adjacent to a vertex at distance $a-1$ from $w$, so $x_a$ must be adjacent to a vertex with code $(a-1,a)$, namely $x_{a-1}$, and cannot be adjacent to $y_{a-1}$ or $z_{a-1}$. Similarly, (R6) follows.

Regarding (R7), if a vertex $y_a$ has code $(a,a)$, then it must be adjacent to a vertex at distance $a-1$ from $u$ and to a (possibly the same or different) vertex at distance $a-1$ from $w$.  So $y_a$ could be adjacent to a vertex with code $(a-1,a-1)$, or to two vertices with codes $(a-1,a)$ and $(a,a-1)$, or to both.

If we let $S_a = \{x_a,y_a,z_a\}$, notice that edges within a set $S_a$ do not affect the codes on the vertices, since their distances from $u$ and $w$ differ by at most one.  However, an edge between $S_a$ and $S_{b}$, where $b \geq a + 2$, would change the codes on the vertices. So, (R8) and (R9) follow.~\hfill
\end{proof}
 
\begin{figure}[ht]
\centering
\begin{tikzpicture}[scale=.5, transform shape]

\node [draw, shape=circle, scale=1.3] (a) at  (-2,1) {};
\node [draw, shape=circle, scale=1.3] (b) at  (-2,-1) {};

\node [draw, shape=circle, scale=1.3] (1) at  (0, 2) {};
\node [draw, shape=circle, scale=1.3] (2) at  (0, 0) {};
\node [draw, shape=circle, scale=1.3] (3) at  (0, -2) {};

\node [draw, shape=circle, scale=1.3] (11) at  (2,2) {};
\node [draw, shape=circle, scale=1.3] (22) at  (2,0) {};
\node [draw, shape=circle, scale=1.3] (33) at  (2,-2) {};

\node [draw, shape=circle, scale=1.3] (111) at  (4,2) {};
\node [draw, shape=circle, scale=1.3] (222) at  (4,0) {};
\node [draw, shape=circle, scale=1.3] (333) at  (4,-2) {};

\node [draw, shape=circle, scale=1.3] (1111) at  (6,2) {};
\node [draw, shape=circle, scale=1.3] (2222) at  (6,0) {};
\node [draw, shape=circle, scale=1.3] (3333) at  (6,-2) {};

\node [scale=1.8] at (-2.6,1) {$u$};
\node [scale=1.8] at (-2.6,-1) {$w$};

\draw(a)--(1)--(11)--(111)--(1111)--(7,2);\draw[dotted, thick](2)--(22)--(222)--(2222)--(7,0);\draw(b)--(3)--(33)--(333)--(3333)--(7,-2);
\draw[dotted, thick](7,2)--(8,2);\draw[dotted, thick](7,0)--(8,0);\draw[dotted, thick](7,-2)--(8,-2);
\draw(a)--(b)--(2)--(a);
\draw[dotted, thick](1)--(2)--(3)--(22)--(1);\draw[dotted, thick](11)--(22)--(33)--(222)--(11);\draw[dotted, thick](111)--(222)--(333)--(2222)--(111);\draw[dotted, thick](1111)--(2222)--(3333);
\draw[dotted, thick](1111)--(7.5,0.5);\draw[dotted, thick](3333)--(7.5,-0.5);
\draw[dotted, thick](1) .. controls(1,0) .. (3);\draw[dotted, thick](11) .. controls(3,0) .. (33);\draw[dotted, thick](111) .. controls(5,0) .. (333);\draw[dotted, thick](1111) .. controls(7,0) .. (3333);

\end{tikzpicture}
\caption{A graph $G$ with $\cdim(G)=2$.}\label{cdim_2}
\end{figure}
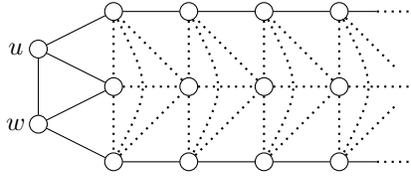

So, Theorem~\ref{thm_cdim2}, together with Observation~\ref{obs}(a), we have the following

\begin{corollary}\label{cor_cdim2}
Let $G$ be a connected graph of order $n \ge 3$. Then $\cdim(G)=2$ if and only if $G$ is not a path and $G \in \mathcal{F}_r$ for some positive integer $r$. See Figure~\ref{cdim_2}, where a solid edge must be present whenever the two vertices incident to the solid edge are in the graph, but a dotted edge is not necessarily present. Thus, if $\cdim(G)=2$, then $G$ is planar.
\end{corollary}

Next, we show that there exists a non-planar graph $G$ such that $\cdim(G)=k \ge 3$. 

\begin{remark}\label{rem_cdim3}
For an integer $k \ge 3$, there exists a non-planar graph $G$ satisfying $\cdim(G)=k$. Let $H$ be a complete graph on $k+2$ vertices with vertex set $V(H)=\cup_{i=1}^{k+2}\{u_i\}$, and let $G$ be the graph obtained from $H$ by subdividing the edge $u_1u_{k+2}$ exactly once (see Figure~\ref{cdim_3} when $k=3$). Since $G$ contains $K_5$ as a minor, $G$ is not planar by Theorem~\ref{planarity}. We will show that $\cdim(G)=k$. Let $S$ be any minimum resolving set of $G$. Then $|S \cap \{u_1, u_{k+2}\}| \ge 1$ since $u_1$ and $u_{k+2}$ are twin vertices in $G$, and $|S \cap (\cup_{i=2}^{k+1}\{u_i\})| \ge k-1$ since any two vertices in $\cup_{i=2}^{k+1}\{u_i\}$ are twin vertices in $G$. So, $\cdim(G) \ge \dim(G) \ge k$ by Observation~\ref{obs}(a). On the other hand, $\cup_{i=1}^{k}\{u_i\}$ forms a connected resolving set of $G$, and thus $\cdim(G) \le k$. 
\end{remark}

\begin{figure}[ht]
\centering
\begin{tikzpicture}[scale=.5, transform shape]

\node [draw, shape=circle, scale=1.3] (1) at  (0,2.7) {};
\node [draw, shape=circle, scale=1.3] (2) at  (-2.3, 1.5) {};
\node [draw, shape=circle, scale=1.3] (6) at  (2.3, 1.5) {};
\node [draw, shape=circle, scale=1.3] (3) at  (-2.3, -1.5) {};
\node [draw, shape=circle, scale=1.3] (5) at  (2.3, -1.5) {};
\node [draw, shape=circle, scale=1.3] (4) at  (0,-1.5) {};

\node [scale=1.8] at (0,3.2) {$u_3$};
\node [scale=1.8] at (-3,1.5) {$u_2$};
\node [scale=1.8] at (-3,-1.5) {$u_1$};
\node [scale=1.8] at (3,-1.5) {$u_5$};
\node [scale=1.8] at (3,1.5) {$u_4$};

\draw(1)--(2)--(3)--(4)--(5)--(6)--(1); \draw(1)--(3)--(6)--(2)--(5)--(1);

\end{tikzpicture}
\caption{A non-planar graph $G$ with $\cdim(G)=3$.}\label{cdim_3}
\end{figure}
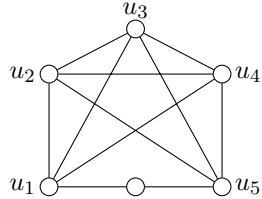

Let $G$ be a connected graph with $\dim(G)=2$, and let $\mathfrak{C}$ be the collection of all minimum resolving sets of $G$. Let $d:=\min_{S \in\mathfrak{C}}\{d(x,y): x,y \in S\}$. Note that $d=1$, equivalently $\cdim(G)=2$, implies $G$ is planar (see Corollary~\ref{cor_cdim2}). Given that $\dim(G)=2$ does not imply planarity of $G$, it is a natural question to consider values of $d$ for which planarity of $G$ is guaranteed. We will show that $d=2$ fails to imply planarity of $G$.

\begin{remark}
Let $H=K_{3,3}$ such that $V(H)$ is partitioned into two partite sets $V_1=\{u_1,u_2,u_3\}$ and $V_2=\{w_1,w_2,w_3\}$. Let $G$ be the graph obtained from $H$ by subdividing the edge $u_2w_2$ exactly twice and subdividing the edges $u_3w_1$ and $u_3w_2$ exactly once (see Figure~\ref{d2}). One can easily check that $\{u_2,u_3\}$ forms a resolving set of $G$, and thus $\dim(G)=2$ by Theorem~\ref{dim_characterization}(a). Since $d_G(u_2,u_3)=2$ and $G$ is not planar, $d=2$ by Corollary~\ref{cor_cdim2}. 
\end{remark}

\begin{figure}[ht]
\centering
\begin{tikzpicture}[scale=.5, transform shape]

\node [draw, shape=circle, scale=1.3] (3) at  (0, 6) {};
\node [draw, shape=circle, scale=1.3] (2) at  (0, 3) {};
\node [draw, shape=circle, scale=1.3] (1) at  (0, 0) {};
\node [draw, shape=circle, scale=1.3] (6) at  (6, 6) {};
\node [draw, shape=circle, scale=1.3] (5) at  (6, 3) {};
\node [draw, shape=circle, scale=1.3] (4) at  (6,0) {};

\node [draw, shape=circle, scale=1.3] (7) at  (1.8, 3) {};
\node [draw, shape=circle, scale=1.3] (8) at  (4.2, 3) {};

\node [draw, shape=circle, scale=1.3] (9) at  (1.3, 4.7) {};
\node [draw, shape=circle, scale=1.3] (10) at  (2, 5) {};

\node [scale=1.8] at (-0.7,0) {$u_1$};
\node [scale=1.8] at (-0.7,3) {$u_2$};
\node [scale=1.8] at (-0.7,6) {$u_3$};
\node [scale=1.8] at (6.7,0) {$w_1$};
\node [scale=1.8] at (6.7,3) {$w_2$};
\node [scale=1.8] at (6.7,6) {$w_3$};

\draw(6)--(1)--(5); \draw(6)--(2)--(7)--(8)--(5);\draw(6)--(3)--(10)--(5); \draw(3)--(9)--(4)--(2);\draw(1)--(4);

\end{tikzpicture}
\caption{A non-planar graph $G$ with $\dim(G)=2$ and $d=2$.}\label{d2}
\end{figure}
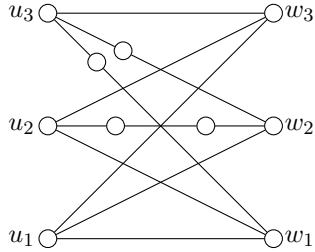


\section{The connected metric dimension of some graphs}

In this section, we determine $\cdim(G)$ and $\cdim_G(v)$, for any $v \in V(G)$, when $G$ is a tree, the Petersen graph, a wheel graph, a bouquet of $m$ cycles for $m \ge 2$, a complete multi-partite graph, or a grid graph. Along the way, we show that $\cdim(G)-\dim(G)$ can be arbitrarily large.\\

First, we consider trees. We recall some terminologies and notations. Fix a tree $T$. An \emph{exterior degree-two vertex} is a vertex of degree two that lies on a path from a terminal vertex to its major vertex, and an \emph{interior degree-two vertex} is a vertex of degree two such that the shortest path to any terminal vertex includes a major vertex. Let $\mathcal{M}(T)$ be the set of exterior major vertices of $T$. Let $\mathcal{M}_1(T)=\{w \in \mathcal{M}(T): ter_T(w)=1\}$ and let $\mathcal{M}_2(T)=\{w \in \mathcal{M}(T): ter_T(w) \ge 2\}$; note that $\mathcal{M}(T)=\mathcal{M}_1(T) \cup \mathcal{M}_2(T)$. For any vertex $v\in \mathcal{M}(T)$, let $T_v$ be the subtree of $T$ induced by $v$ and all vertices belonging to the paths joining $v$ with its terminal vertices. Let $\mathcal{T}_1=\{T_v: v \in \mathcal{M}_1(T)\}$ and $\mathcal{T}_2=\{T_v: v \in \mathcal{M}_2(T)\}$, and let $V(\mathcal{T}_1)= \cup_{T'\in\mathcal{T}_1}V(T')$ and $V(\mathcal{T}_2)=\cup_{T'' \in \mathcal{T}_2}V(T'')$.

\begin{theorem}\emph{\cite{tree3}}\label{zhang_tree}
Let $T$ be a tree with $ex(T)=k \ge 1$, and let $v_1, v_2, \ldots, v_k$ be the exterior major vertices of $T$. For each $i$ ($1 \le i \le k$), let $\ell_{i,1}, \ell_{i,2}, \ldots, \ell_{i, \sigma_i}$ be the terminal vertices of $v_i$ with $ter_T(v_i)=\sigma_i \ge 1$, and let $P_{i,j}$ be the $v_i-\ell_{i,j}$ path, where $1 \le j \le \sigma_i$. Let $W \subseteq V(T)$. Then $W$ is a minimum resolving set of $T$ if and only if $W$ contains exactly one vertex from each of the paths $P_{i,j}-v_i$ ($1 \le j \le \sigma_i$ and $1 \le i\le k$) with exactly one exception for each $i$ with $1 \le i \le k$ and $W$ contains no other vertices of $T$.
\end{theorem}

As a consequence of Theorem~\ref{zhang_tree}, we have the following

\begin{corollary}\label{obs_tree}
Let $T$ be a tree with $ex(T)=k \ge 1$, and let $v_1, v_2, \ldots, v_k$ be the exterior major vertices of $T$. For each $i$ ($1 \le i \le k$), let $ter_T(v_i)=\sigma_i \ge 1$ and let $N(v_i) \cap V(T_{v_i})=\{s_{i,1}, s_{i,2}, \ldots, s_{i, \sigma_i}\}$. Then 
\begin{itemize}
\item[(a)] $W \cap \{v_1, v_2, \ldots, v_k\}=\emptyset$ for any minimum resolving set $W$ of $T$; 
\item[(b)] $\cup_{i=1}^{k}(N(v_i)\cap V(T_{v_i})-\{s_{i,1}\})$ forms a minimum resolving set of $T$.
\end{itemize}
\end{corollary}

Now, we consider $\cdim(T)$ and $\cdim_T(v)$ for a tree $T$ and $v \in V(T)$.

\begin{lemma}\label{lemma_tree1}
Let $T$ be a tree with $ex(T)=1$. Then $\cdim(T)=\sigma(T)$.
\end{lemma}

\begin{proof}
Let $v$ be the (exterior) major vertex of $T$ with $ter_T(v)=\sigma$; then $\sigma \ge 3$. Let $N(v)=\{s_1, s_2, \ldots, s_{\sigma}\}$. By Corollary~\ref{obs_tree}(b) and Theorem~\ref{dim_tree}, $S=N(v)-\{s_1\}$ forms a minimum resolving set of $T$ with $|S|=\sigma-1=\dim(T)$. By Theorem~\ref{zhang_tree} and Corollary~\ref{obs_tree}(a), for any minimum resolving set $S'$ of $T$, $T[S']$ is disconnected; thus $\cdim(T) \ge \dim(T)+1=\sigma(T)$. On the other hand, $S^*=N[v] -\{s_1\}$ forms a connected resolving set of $T$ with $|S^*|=\sigma(T)$, and hence $\cdim(T) \le \sigma(T)$. Thus, $\cdim(T)=\sigma(T)$.~\hfill
\end{proof}

\begin{theorem}\label{cdim_tree}
Let $T$ be a tree of order at least two, and let $\mathcal{D}$ denote the set of vertices consisting of interior degree-two vertices and major vertices of terminal degree zero in $T$. Then
\begin{itemize}
\item[(a)]
\begin{equation*}
\cdim(T)=\left\{
\begin{array}{ll}
1 & \mbox{ if $T$ is a path},\\
|\mathcal{D}|+\sigma(T) & \mbox{ if $T$ is not a path};
\end{array}\right.
\end{equation*}
\item[(b)] if we let $\Gamma=\cup_{w \in \mathcal{M}_2(T)}(N[w] \cap V(T_w)) \cup \mathcal{M}_1(T) \cup \mathcal{D}$, then, for any $v \in V(T)$,
\begin{equation*}
\cdim_T(v)=\left\{
\begin{array}{ll}
1 & \mbox{if } \mathcal{M}(T)=\emptyset \mbox{ and $v$ is an end vertex},\\
2 & \mbox{if } \mathcal{M}(T)=\emptyset \mbox{ and $v$ is not an end vertex},\\
|\mathcal{D}|+\sigma(T) & \mbox{if } \mathcal{M}(T)\neq\emptyset \mbox{ and } v \in \Gamma,\\
|\mathcal{D}|+\sigma(T)+d(v,u) & \mbox{if } \mathcal{M}(T)\neq\emptyset \mbox{ and } v \in V(\mathcal{T}_1)-\Gamma,\\
|\mathcal{D}|+\sigma(T)+d(v,u)-1 & \mbox{if } \mathcal{M}(T)\neq\emptyset \mbox{ and } v \in V(\mathcal{T}_2)-\Gamma,
\end{array}
\right.
\end{equation*}
where $d(v,u)$ denotes the distance between $v$ and its nearest major vertex $u$.
\end{itemize}
\end{theorem}

\begin{proof}
(a) If $ex(T)=0$, then $T$ is a path and $\cdim(T)=1$ by Observation~\ref{obs}(a). If $ex(T)=1$, then $\mathcal{D}=\emptyset$ and $\cdim(T)=\sigma(T)$ by Lemma~\ref{lemma_tree1}. So, suppose that $ex(T) \ge 2$; let $\mathcal{M}_2(T)=\{u_1, u_2, \ldots, u_k\}$ with $ter_T(u_i)=\sigma_i \ge 2$ and $N(u_i) \cap V(T_{u_i})=\{s_{i,1}, s_{i,2}, \ldots, s_{i, \sigma_i}\}$, where $1 \le i \le k$. Let $S=\cup_{i=1}^{k}(N(u_i)\cap V(T_{u_i})-\{s_{i,1}\})$. If $\mathcal{D}=\emptyset$ and $\mathcal{M}_1(T)=\emptyset$, then $\cdim(T)=\sigma(T)$ since $S \cup \mathcal{M}_2(T)$ forms a minimum connected resolving set of $T$ by Theorem~\ref{zhang_tree} and Corollary~\ref{obs_tree}. So, suppose that $\mathcal{D} \neq\emptyset$ or $\mathcal{M}_1(T) \neq \emptyset$; then, for each vertex $x \in \mathcal{D} \cup \mathcal{M}_1(T)$, there exist two distinct vertices $w_1,w_2 \in \mathcal{M}_2(T)$ such that $x$ lies on the $w_1-w_2$ path. By Theorem~\ref{zhang_tree} and Corollary~\ref{obs_tree} and the fact that $T[S \cup \mathcal{M}_2(T)]$ is disconnected, $\cdim(T) \ge  (\sum_{i=1}^{k} ter_T(u_i))+|\mathcal{M}_1(T)|+|\mathcal{D}|= (\sum_{u \in \mathcal{M}(T)} ter_T(u)) + |\mathcal{D}|= \sigma(T) + |\mathcal{D}|$. On the other hand, $S'=\cup_{i=1}^{k}(N[u_i] \cap V(T_{u_i})-\{s_{i,1}\}) \cup \mathcal{M}_1(T) \cup \mathcal{D}$ forms a connected resolving set of $T$ with $|S'|=\sigma(T)+|\mathcal{D}|$, and thus $\cdim(T) \le |S'|=\sigma(T)+|\mathcal{D}|$. In each case, $\cdim(T)=\sigma(T)+|\mathcal{D}|$ when $ex(T) \ge 2$.\\

(b) Let $v \in V(T)$; then $v\in V(\mathcal{T}_1) \cup V(\mathcal{T}_2) \cup \mathcal{D}$. Let $\mathfrak{C}$ be the collection of all minimum connected resolving sets of $T$.

\textbf{Case 1: $\mathcal{M}(T)=\emptyset$.} In this case, $T$ is a path. If $v$ is an end vertex, then $\cdim_T(v)=1$ by Theorem~\ref{cdim_characterization}(a). If $v$ is not an end vertex, say $vw \in E(T)$, then $\{v,w\}$ forms a connected resolving set of $T$, and thus $\cdim_T(v)=2$ by Theorem~\ref{cdim_characterization}(a).

\textbf{Case 2: $\mathcal{M}(T)\neq\emptyset$.} In this case, $\mathcal{M}_2(T)\neq\emptyset$. Let $\mathcal{M}_2(T)=\{w_1,w_2, \ldots, w_k\}$ with $ter_T(w_i)=\sigma_i \ge 2$ and $N(w_i) \cap V(T_{w_i})=\{s_{i,1}, s_{i,2}, \ldots, s_{i, \sigma_i}\}$, where $1 \le i \le k$. Let $S^*=\cup_{i=1}^{k}(N[w_i] \cap V(T_{w_i})-\{s_{i,1}\}) \cup \mathcal{M}_1(T) \cup \mathcal D$ and $S'=(S^*-\cup_{i=1}^{k}\{s_{i,2}\}) \cup (\cup_{j=1}^{k}\{s_{j,1}\})$. Note that both $S^*$ and $S'$ are  connected resolving sets of $T$ with $|S^*|=|S'|=\sigma(T)+|\mathcal{D}|=\cdim(T)$ by Corollary~\ref{obs_tree} and (a) of the present theorem; thus $S^*,S' \in\mathfrak{C}$.

First, let $v\in \Gamma$; then $v\in S^*$ or $v \in S'$. Since there exists a minimum connected resolving set of $T$ containing $v$, $\cdim_T(v)=\cdim(T)=\sigma(T)+|\mathcal{D}|$. 

Second, let $v \in V(\mathcal{T}_1)-\Gamma$; then $v \in V(T_u)-\{u\}$ for some $u \in \mathcal{M}_1(T)$. Let $Q$ be the $u-v$ path, and let $R_1 = S^* \cup V(Q)$; note that $S^* \cap V(Q)=\{u\}$. Since $R_1$ is a connected resolving set of $T$ containing $v$, $\cdim_T(v) \le |R_1|=|S^*|+|V(Q)|-1=\sigma(T)+|\mathcal{D}|+d(v,u)$. On the other hand, $d(v, R) \ge d(v,u)$ for any $R \in \mathfrak{C}$, and hence $\cdim_T(v) \ge |S^*|+d(v,u)=\sigma(T)+|\mathcal{D}|+d(v,u)$. Thus, $\cdim_T(v)=\sigma(T)+|\mathcal{D}|+d(v,u)$.

Third, let $v \in V(\mathcal{T}_2)-\Gamma$; then $v \in V(T_u)-N[u]$ for some $u \in \mathcal{M}_2(T)$. Let $Q'$ be the $u-v$ path, and let $t\in N(u) \cap V(T_u)$ lie on the $u-v$ path. Let $S \in \mathfrak{C}$ with $|S|=\cdim(T)$ such that $t \in S$. Let $R_2=S \cup V(Q')$; then $S \cap V(Q')=\{u,t\}$. Since $R_2$ is a connected resolving set of $T$ containing $v$, $\cdim_T(v) \le |R_2|=|S|+|V(Q')|-2=\sigma(T)+|\mathcal{D}|+d(v,u)-1$. On the other hand, $d(v, R) \ge d(v,t)=d(v,u)-1$  for any $R \in \mathfrak{C}$, and hence $\cdim_T(v) \ge |S|+d(v,u)-1=\sigma(T)+|\mathcal{D}|+d(v,u)-1$. Thus, $\cdim_T(v)=\sigma(T)+|\mathcal{D}|+d(v,u)-1$.~\hfill
\end{proof}

Next, we consider the Petersen graph.

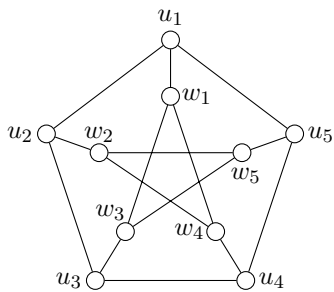
\begin{figure}[ht]
\centering
\begin{tikzpicture}[scale=.5, transform shape]

\node [draw, shape=circle, scale=1.3] (1) at  (0,3.5) {};
\node [draw, shape=circle, scale=1.3] (2) at  (-3.3, 1) {};
\node [draw, shape=circle, scale=1.3] (3) at  (-2, -2.9) {};
\node [draw, shape=circle, scale=1.3] (4) at  (2, -2.9) {};
\node [draw, shape=circle, scale=1.3] (5) at  (3.3, 1) {};

\node [draw, shape=circle, scale=1.3] (11) at  (0,2) {};
\node [draw, shape=circle, scale=1.3] (22) at  (-1.9, 0.5) {};
\node [draw, shape=circle, scale=1.3] (33) at  (-1.2, -1.6) {};
\node [draw, shape=circle, scale=1.3] (44) at  (1.2, -1.6) {};
\node [draw, shape=circle, scale=1.3] (55) at  (1.9, 0.5) {};

\node [scale=1.8] at (0,4.1) {$u_1$};
\node [scale=1.8] at (-4,1) {$u_2$};
\node [scale=1.8] at (-2.7,-2.9) {$u_3$};
\node [scale=1.8] at (2.7,-2.9) {$u_4$};
\node [scale=1.8] at (4,1) {$u_5$};

\node [scale=1.8] at (0.7,2) {$w_1$};
\node [scale=1.8] at (-1.9,1) {$w_2$};
\node [scale=1.8] at (-1.6,-1.1) {$w_3$};
\node [scale=1.8] at (0.5,-1.6) {$w_4$};
\node [scale=1.8] at (2,-0.1) {$w_5$};

\draw(1)--(2)--(3)--(4)--(5)--(1);
\draw(11)--(44)--(22)--(55)--(33)--(11);
\draw(1)--(11);\draw(2)--(22);\draw(3)--(33);\draw(4)--(44);\draw(5)--(55);

\end{tikzpicture}
\caption{The Petersen graph $\mathcal{P}$ and its labeling.}\label{Petersen}
\end{figure}

\begin{proposition}
Let $\mathcal{P}$ be the Petersen graph. Then $\cdim(\mathcal{P})=4=\cdim_{\mathcal{P}}(v)$ for any $v\in V(\mathcal{P})$.
\end{proposition}

\begin{proof}
Let $\mathcal{P}$ be the Petersen graph with its labeling as in Figure~\ref{Petersen}. We note that $\mathcal{P}$ is vertex-transitive and edge-transitive (see~\cite{petersen}); thus, $\cdim(\mathcal{P})=\cdim_{\mathcal{P}}(v)$, for any $v\in V(\mathcal{P})$, by Observation~\ref{obs_transitive}. By Theorem~\ref{dim_Petersen} and Observation~\ref{obs}(a), $\cdim(\mathcal{P}) \ge \dim(\mathcal{P})=3$. Since $\{u_1, w_1, w_4, w_2\}$ forms a connected resolving set of $\mathcal{P}$, $\cdim(\mathcal{P}) \le 4$. We will show that $\cdim(\mathcal{P}) \ge 4$. Assume, to the contrary, that $\cdim(\mathcal{P})=3$. For any minimum connected resolving set $S$ of $\mathcal{P}$, we may assume that $X=\{u_1, w_1\} \subset S$ by edge-transitivity of $\mathcal{P}$. Then $\code_X(u_3)=\code_X(u_4)=\code_X(w_2)=\code_X(w_5)$. Since $\diam(\mathcal{P})=2$, any vertex in $V(\mathcal{P})-X$ can distinguish at most three vertices of $\{u_3, u_4, w_2, w_5\}$ by distances; thus $\cdim(\mathcal{P}) \neq 3$. Therefore, by Observation~\ref{obs_transitive}, $\cdim(\mathcal{P})=4=\cdim_{\mathcal{P}}(v)$ for any $v\in V(\mathcal{P})$.~\hfill
\end{proof}

\begin{porism}
If $B$ is a minimum resolving set of the Petersen graph $\mathcal{P}$, then $\mathcal{P}[B]$ is an empty graph.
\end{porism}

Next, we consider wheel graphs. We begin by recalling the following

\begin{observation}\emph{\cite{wheel1}}\label{obs_wheel}
For $n \ge 7$, let $W_n$ be the wheel graph on $n+1$ vertices, and let $w$ be the central vertex of $W_n$. Then, for any minimum resolving set $S$ of $W_n$, $w \not\in S$ and $G[S]$ is disconnected.
\end{observation}

\begin{proposition}\label{cdim_wheel}
For $n \ge 3$, let $W_n$ be the wheel graph on $n+1$ vertices, and let $w$ be the central vertex of $W_n$. Then, for any $v\in V(W_n)$,
\begin{itemize}
\item[(a)] $\cdim(W_3)=3=\cdim_{W_3}(v)$;
\item[(b)] for $n \in \{4,5\}$, $\cdim(W_n)=2=\cdim_{W_n}(v)$ for $v\neq w$, and $\cdim_{W_n}(w)=3$;
\item[(c)] for $n \ge 6$, $\cdim(W_n)=\lfloor\frac{2n+2}{5}\rfloor+1=\cdim_{W_n}(v)$.
\end{itemize}
\end{proposition}

\begin{proof}
Let $W_n=C_n+K_1$ be the wheel graph on $n+1$ vertices, where $C_n$ is given by $u_1, u_2, \ldots, u_n, u_1$ and $w$ is the central vertex of $W_n$; note that the vertices on the $C_n$ belong to the same orbit under the automorphism group of $W_n$. Let $S$ be an arbitrary minimum resolving set of $W_n$.

First, let $n=3$; then $W_3=K_4$. By Observation~\ref{obs_transitive} and Theorem~\ref{cdim_characterization}(b), $\cdim(W_3)=3=\cdim_{W_3}(v)$ for any vertex $v \in V(W_3)$.

Second, let $n \in \{4,5\}$; then $\dim(W_n)=2$ by Theorem~\ref{dim_wheel}. So, $w \not\in S$. Since $S'=\{u_1,u_2\}$ forms a minimum connected resolving set of $W_n$, $\cdim(W_n)=2=\cdim_{W_n}(v)$ for $v\neq w$ and $\cdim_{W_n}(w)=3$.

Third, let $n=6$; then $\dim(W_6)=3$ by Theorem~\ref{dim_wheel}. Since both $S_1=\{u_1,u_2,u_3\}$ and $S_2=\{u_1, w, u_3\}$ form minimum connected resolving sets of $W_6$, $\cdim(W_6)=3=\cdim_{W_6}(v)$ for any $v\in V(W_6)$ by Observations~\ref{obs}(c) and~\ref{r=dim}.

Next, let $n \ge 7$. By Observation~\ref{obs_wheel}, $w\not\in S$ and $G[S]$ is disconnected; thus, $\cdim(W_n) \ge \dim(W_n)+1=\lfloor\frac{2n+2}{5}\rfloor+1$ by Theorem~\ref{dim_wheel}. On the other hand, $S^*= S \cup \{w\}$ forms a connected resolving set of $W_n$, and thus $\cdim(W_n) \le |S^*|=|S|+1=\lfloor\frac{2n+2}{5}\rfloor+1$. So, $\cdim(W_n)=\lfloor\frac{2n+2}{5}\rfloor+1=\cdim_{W_n}(v)$ for any $v \in V(W_n)$.~\hfill
\end{proof}

Next, we consider a bouquet of cycles. Throughout this section, let $B_m$ be a bouquet of $m$ cycles (i.e., the vertex sum of $m$ cycles at one common vertex), where $m \ge 2$, and let $w$ be the cut-vertex of $B_m$. Let $C^1, C^2, \ldots, C^m$ be the $m$ cycles of $B_m$. For each $i \in \{1,2,\ldots, m\}$, let $P^i=C^i-w$ and let $N(w) \cap V(P^i)=\{u_{i,1}, u'_{i,1}\}$.

\begin{lemma}\emph{\cite{bouquet}}\label{lemma_bouquet}
Let $S$ be any resolving set of $B_m$. Then 
\begin{itemize}
\item[(a)] for each $i \in \{1,2,\ldots, m\}$, $|S \cap V(P^i)| \ge 1$; 
\item[(b)] for any two distinct even cycles $C^{i}$ and $C^{j}$ of $B_m$, $|S \cap (V(P^{i}) \cup V(P^{j}))| \ge 3$.
\end{itemize}
\end{lemma}

Based on Theorem~\ref{dim_bouquet} and Lemma~\ref{lemma_bouquet}, we have the following 

\begin{observation}\label{obs_bouquet}
Let $w$ be the cut-vertex of $B_m$. Then
\begin{itemize}
\item[(a)] for any minimum resolving set $S$ of $B_m$, $w \not\in S$;
\item[(b)] for any minimum connected resolving set $S_c$ of $B_m$, $w \in S_c$.
\end{itemize}
\end{observation}

\begin{lemma}\label{lemma_cdim_bouquet}
Let $S_c$ be any minimum connected resolving set of $B_m$. Then 
\begin{itemize}
\item[(a)] for each $i \in \{1,2,\ldots, m\}$, $|S_c \cap V(P^i)| \ge 1$; 
\item[(b)] for any two distinct cycles $C^{i}$ and $C^{j}$ of length at least four in $B_m$, $|S_c \cap (V(P^{i}) \cup V(P^{j}))| \ge 3$.
\end{itemize}
\end{lemma}

\begin{proof}
(a) If $|S_c \cap V(P^i)|=0$ for some $i \in \{1,2,\ldots, m\}$, then $\code_{S_c}(u_{i,1})=\code_{S_c}(u'_{i,1})$, contradicting the assumption that $S_c$ is a resolving set of $B_m$. So, $|S_c \cap V(P^i)| \ge 1$ for each $i \in \{1,2,\ldots, m\}$.

(b) Let $C^i$ and $C^j$ be two distinct cycles of length at least four in $B_m$. Assume, to the contrary, that $|S_c \cap (V(P^{i}) \cup V(P^{j}))| \le 2$. By (a) of the current lemma, we have $|S_c \cap V(P^i)|=1$ and $|S_c \cap V(P^j)|=1$. Notice $w \in S_c$. By the connectedness of $B_m[S_c]$, we may assume, without loss of generality, that $S_c \cap V(P^i)=\{u_{i,1}\}$ and $S_c \cap V(P^j)=\{u_{j,1}\}$. Then $\code_{S_c}(u'_{i,1})=\code_{S_c}(u'_{j,1})$ contradicting the assumption that $S_c$ is a resolving set of $B_m$. Thus, $|S_c \cap (V(P^{i}) \cup V(P^{j}))| \ge 3$.~\hfill
\end{proof}

\begin{proposition}\label{cdim_bouquet}
Let $b$ be the number of the cycles isomorphic to $C_3$ in $B_m$. Then 
\begin{equation*}
\cdim(B_m)=\left\{
\begin{array}{ll}
m+1 & \mbox{ if } b=m,\\
2m-b & \mbox{ if } b \neq m. 
\end{array}\right.
\end{equation*}
\end{proposition}

\begin{proof}
Let $w$ be the cut-vertex of $B_m$ for $m \ge 2$. Let $C^1, C^2, \ldots, C^b$ be the cycles isomorphic to $C_3$, and let $C^{b+1}, C^{b+2}, \ldots, C^m$ be the cycles of length at least four in $B_m$; let $b=0$ if each cycle of $B_m$ is of length at least four, and let $b=m$ if each cycle of $B_m$ is isomorphic to $C_3$.\\

\textbf{Case 1: $b=m$.} By Observation~\ref{obs_bouquet}(b) and Lemma~\ref{lemma_cdim_bouquet}(a), $\cdim(B_m) \ge m+1$. Since $\{w\} \cup (\cup_{i=1}^{m}\{u_{i,1}\})$ forms a connected resolving set of $B_m$, $\cdim(B_m) \le m+1$. Thus, $\cdim(B_m)=m+1$.\\

\textbf{Case 2: $b<m$.} First, we make the following\\

\textbf{Claim.} $R=(\cup_{i=1}^{b}\{u_{i,1}\}) \cup (\cup_{i=b+1}^{m-1}\{u_{i,1}, u'_{i,1}\}) \cup \{w,u_{m,1}\}$ forms a connected resolving set of $B_m$ with $|R|=2m-b$; thus $\cdim(B_m) \le 2m-b$.\\

\textit{Proof.} Since $B_m[R]$ is connected, it suffices to show that $R$ is a resolving set of $B_m$. Let $x$ and $y$ be two distinct vertices in $B_m$. If $x, y \in V(C^i)$ for some $i\in\{1,2,\ldots, m\}$, then $R \cap V(C^i)$ resolves $x$ and $y$, since $R \cap V(C^i)$ contains two adjacent vertices of $C^i$. So, suppose that $x \in V(P^i)$ and $y \in V(P^j)$ for $i \neq j$. If $d(w, x) \neq d(w, y)$, then $w \in R$ resolves $x$ and $y$. So, let $d(w,x)=d(w,y)$. If $C^i$ or $C^j$ is isomorphic to $C_3$, say the former, then $d(u_{i,1}, x) \le 1 < 2 \le d(u_{i,1}, y)$. If neither $C^i$ nor $C^j$ is isomorphic to $C_3$, then $|R \cap V(P^i)|=2$ or $|R \cap V(P^j)|=2$, say the former, i.e., $R \cap V(P^i)=\{u_{i, 1}, u'_{i,1}\}$. We may assume that $u_{i,1}$ lies on the $w-x$ geodesic by a relabeling of the vertices if necessary. Then $d(u_{i,1}, x)=d(w,x)-1<d(w,x)+1=d(u_{i,1},y)$.~$\Box$\\

Second, we will show that $\cdim(B_m) \ge 2m-b$. If $b=m-1$, then $\cdim(B_m) \ge m+1=2m-b$ by Observation~\ref{obs_bouquet}(b) and Lemma~\ref{lemma_cdim_bouquet}(a). If $b\leq m-2$, i.e., there are at least two cycles of length at least four in $B_m$, then $\cdim(B_m) \ge b+2(m-1-b)+1+1=2m-b$ by Observation~\ref{obs_bouquet}(b) and Lemma~\ref{lemma_cdim_bouquet}.~\hfill
\end{proof}

If every cycle in $B_m$ is isomorphic to $C_3$, then $N[w]=V(B_m)$. Also, for $x \in V(P^i)$ with $d(x,w) <\diam(C^i)$, if $y \in N(w)$ and $y$ does not lie on a $w-x$ geodesic, then $d(x,y)=d(x,w)+1$. These observations, together with the proof of Proposition~\ref{cdim_bouquet}, yield the following

\begin{corollary}\label{cdim_v}
Let $w$ be the cut-vertex of $B_m$ for $m \ge 2$, and let $b$ be the number of the cycles isomorphic to $C_3$ in $B_m$. For each $j \in\{1,2,\ldots,m\}$, let $\Gamma_j=V(C^j)-N[w]$ and $D_j=\{v\in V(C^j): d(v,w)=\diam(C^j)\}$. Let $v \in V(C^i) \subset V(B_m)$ for some $i \in \{1,2,\ldots,m\}$. Then
\begin{equation*}
\cdim_{B_m}(v)=\left\{
\begin{array}{cl}
\cdim(B_m) & \mbox{ if } v \in N[w],\\
\cdim(B_m)+d(v,w)-1 & \mbox{ if } v\in \Gamma_i \mbox{ and } b=m-1, \mbox{ or } v\in \Gamma_i-D_i \mbox{ and } b \le m-2,\\
\cdim(B_m)+d(v,w)-2 & \mbox{ if } v\in \Gamma_i \cap D_i \mbox{ and } b \le m-2.
\end{array}
\right.
\end{equation*}
\end{corollary}

\begin{remark}\label{non_transitive}
There exists a non vertex-transitive graph that satisfies $\cdim_G(v)=\cdim(G)$ for any vertex $v \in V(G)$. Let $B_m$ be the bouquet of $m$ cycles such that each cycle is isomorphic to $C_3$ (see Figure~\ref{B4} for $B_4$ with $C^i=C_3$ for each $i \in \{1,2,3,4\}$). By Corollary~\ref{cdim_v}, $\cdim_{B_m}(v)=\cdim(B_m)$ for any $v \in V(B_m)$.
\end{remark}

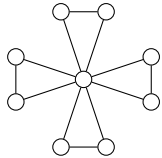
\begin{figure}[ht]
\centering
\begin{tikzpicture}[scale=.6, transform shape]

\node [draw, shape=circle] (0) at  (0,0) {};
\node [draw, shape=circle] (1) at  (-0.5,1.5) {};
\node [draw, shape=circle] (2) at  (0.5,1.5) {};
\node [draw, shape=circle] (3) at  (1.5,0.5) {};
\node [draw, shape=circle] (4) at  (1.5,-0.5) {};
\node [draw, shape=circle] (5) at  (-0.5,-1.5) {};
\node [draw, shape=circle] (6) at  (0.5,-1.5) {};
\node [draw, shape=circle] (7) at  (-1.5,0.5) {};
\node [draw, shape=circle] (8) at  (-1.5,-0.5) {};

\draw(0)--(1)--(2)--(0);\draw(0)--(3)--(4)--(0);
\draw(0)--(5)--(6)--(0);\draw(0)--(7)--(8)--(0);

\end{tikzpicture}
\caption{$B_4$ with $C^i=C_3$ for each $i \in \{1,2,3,4\}$.}\label{B4}
\end{figure}

\begin{remark}\label{cdim_dim}
For each integer $k \ge 1$, there exists a graph $G$ such that $\cdim(G)-\dim(G)=k$. Let $G$ be a bouquet of $k+1$ cycles consisting of one cycle that is isomorphic to $C_3$ and $k$ odd cycles of length at least five. Then $\cdim(G)=2k+1$ by Proposition~\ref{cdim_bouquet} and $\dim(G)=k+1$ by Theorem~\ref{dim_bouquet}. So, $\cdim(G)-\dim(G)=k$, which can be arbitrarily large as $k \rightarrow \infty$. 
\end{remark}

Next, we consider complete multi-partite graphs.

\begin{proposition}\label{cdim_multipartite}
For $k \ge 2$, let $G=K_{a_1, a_2, \ldots, a_k}$ be a complete $k$-partite graph of order $n=\sum_{i=1}^{k}a_i\geq 4$, and let $s$ be the number of partite sets of $G$ consisting of exactly one element. Then
\begin{itemize}
\item[(a)]
$\cdim(G)=\left\{
\begin{array}{ll}
n-1 & \mbox{ if } s=1 \mbox{ and } k=2,\\
\dim(G)=n-k & \mbox{ if } s=0, \\
\dim(G)=n+s-k-1  & \mbox{ otherwise};
\end{array}
\right.
$
\item[(b)] let $V(G)$ be partitioned by $V_1, V_2, \ldots, V_k$ such that $|V_i|=a_i$, $1\leq i \leq k$, and $a_1 \le a_2 \le \ldots \le a_k$; then, for $v \in V(G)$,
$$\cdim_G(v)=\left\{
\begin{array}{ll}
\cdim(G)+1 & \mbox{ if } s=1, k \ge 3, \mbox{ and } v \in V_1,\\
\cdim(G) & \mbox{ otherwise}.
\end{array}
\right.
$$
\end{itemize}
\end{proposition}

\begin{proof}
Let $V(G)$ be partitioned according to part (b) of the statement, and let $V_i=\{w_{i,1}, w_{i,2}, \ldots, w_{i,a_i}\}$.\\ 

(a) First, let $s=0$; then $a_i \ge 2$ for each $i \in \{1,2,\ldots, k\}$. Notice $S=\cup_{i=1}^{k}(V_i-\{w_{i,1}\})$ forms a connected resolving set of $G$. Since $|S|=n-k=\dim(G)$ by Theorem~\ref{dim_multipartite}, it follows by Observation~\ref{obs}(a) that $\cdim(G)=\dim(G)$.

Second, if $s=1$ and $k=2$, then $G$ is a tree with $ex(G)=1$, since $n\geq 4$. Thus, we have $\cdim(G)=n-1$ by Lemma~\ref{lemma_tree1}. If $s=1$ and $k \ge 3$, then $S=\cup_{i=1}^{k}(V_i-\{w_{i,1}\})=\cup_{i=2}^{k}(V_i-\{w_{i,1}\})$ forms a connected resolving set of $G$. Since $|S|=n-k=\dim(G)$ by Theorem~\ref{dim_multipartite}, it follows by Observation~\ref{obs}(a) that $\cdim(G)=\dim(G)$. 

Third, let $s \ge 2$. Then, $S=(\cup_{i=2}^{s}V_i) \cup (\cup_{j=s+1}^{k}(V_j-\{w_{j,1}\}))$ forms a connected resolving set of $G$. Since $|S|=(s-1)+\sum_{j=s+1}^{k}(a_j-1)=(s-1)+\sum_{j=1}^{k}(a_j-1)=(s-1)+(n-k)=\dim(G)$ by Theorem~\ref{dim_multipartite}, it follows by Observation~\ref{obs}(a) that $\cdim(G)=\dim(G)$.\\

(b) Let $v \in  V(G)$. Let $S$ be an arbitrary minimum connected resolving set of $G$. Notice that all vertices of each $V_i$ are equivalent with respect to inclusion in $S$ by virtue of the actions of the automorphism group on $G$. 

First, let $s=1$ and $k=2$, or $s=0$. Then $S \cap V_i \neq \emptyset$ for each $i$, and we may assume that $v\in S$.  Thus, $\cdim_G(v)=|S|=\cdim(G)$.

Second, let $s=1$ and $k \ge 3$. Notice, in this case, $S\cap V_1=\emptyset$ and $S\cap V_i\neq \emptyset$ for $i>1$. Thus, if $v\notin V_1$, then $\cdim_G(v)=\cdim(G)$. Since $G[S\cup V_1]$ is connected, $\cdim_G(v)=\cdim(G)+1$ if $v\in V_1$. 

Third, let $s \ge 2$. In this case, all vertices in $\cup_{i=1}^{s} V_i$ are equivalent with respect to inclusion in $S$ by virtue of the actions of the automorphism group on $G$. Further, we have $S \cap (\cup_{i=1}^{s} V_i) \neq\emptyset$ and $S\cap V_j \neq \emptyset$ for each $j>s$. Thus $\cdim_G(v)=|S|=\cdim(G)$.~\hfill 
\end{proof}

Next, we consider grid graphs (the Cartesian product of two paths). The \emph{Cartesian product} of two graphs $G$ and $H$, denoted by $G \square H$, is the graph with the vertex set $V(G) \times V(H)$ such that $(u,w)$ is adjacent to $(u',w')$ if and only if either $u=u'$ and $ww' \in E(H)$, or $w=w'$ and $uu' \in E(G)$. See Figure~\ref{grid74} for the labeling of $P_7 \square P_4$.

\begin{figure}[ht]
\centering
\begin{tikzpicture}[scale=.6, transform shape]

\node [draw, shape=circle] (4) at  (0,6) {};
\node [draw, shape=circle] (3) at  (0,4) {};
\node [draw, shape=circle] (2) at  (0,2) {};
\node [draw, shape=circle] (1) at  (0,0) {};
\node [draw, shape=circle] (44) at  (2,6) {};
\node [draw, shape=circle] (33) at  (2,4) {};
\node [draw, shape=circle] (22) at  (2,2) {};
\node [draw, shape=circle] (11) at  (2,0) {};
\node [draw, shape=circle] (444) at  (4,6) {};
\node [draw, shape=circle] (333) at  (4,4) {};
\node [draw, shape=circle] (222) at  (4,2) {};
\node [draw, shape=circle] (111) at  (4,0) {};
\node [draw, shape=circle] (4444) at  (6,6) {};
\node [draw, shape=circle] (3333) at  (6,4) {};
\node [draw, shape=circle] (2222) at  (6,2) {};
\node [draw, shape=circle] (1111) at  (6,0) {};
\node [draw, shape=circle] (44444) at  (8,6) {};
\node [draw, shape=circle] (33333) at  (8,4) {};
\node [draw, shape=circle] (22222) at  (8,2) {};
\node [draw, shape=circle] (11111) at  (8,0) {};
\node [draw, shape=circle] (444444) at  (10,6) {};
\node [draw, shape=circle] (333333) at  (10,4) {};
\node [draw, shape=circle] (222222) at  (10,2) {};
\node [draw, shape=circle] (111111) at  (10,0) {};
\node [draw, shape=circle] (4444444) at  (12,6) {};
\node [draw, shape=circle] (3333333) at  (12,4) {};
\node [draw, shape=circle] (2222222) at  (12,2) {};
\node [draw, shape=circle] (1111111) at  (12,0) {};

\draw(1)--(2)--(3)--(4);
\draw(11)--(22)--(33)--(44);
\draw(111)--(222)--(333)--(444);
\draw(1111)--(2222)--(3333)--(4444);
\draw(11111)--(22222)--(33333)--(44444);
\draw(111111)--(222222)--(333333)--(444444);
\draw(1111111)--(2222222)--(3333333)--(4444444);
\draw(1)--(11)--(111)--(1111)--(11111)--(111111)--(1111111);
\draw(2)--(22)--(222)--(2222)--(22222)--(222222)--(2222222);
\draw(3)--(33)--(333)--(3333)--(33333)--(333333)--(3333333);
\draw(4)--(44)--(444)--(4444)--(44444)--(444444)--(4444444);

\node [scale=1.2] at (-1,0) {$(u_1, w_1)$};
\node [scale=1.2] at (-1,2) {$(u_1, w_2)$};
\node [scale=1.2] at (-1,4) {$(u_1, w_3)$};
\node [scale=1.2] at (-1,6) {$(u_1, w_4)$};

\node [scale=1.2] at (2,-0.4) {$(u_2, w_1)$};
\node [scale=1.2] at (2.8,2.4) {$(u_2, w_2)$};
\node [scale=1.2] at (2.8,4.4) {$(u_2, w_3)$};
\node [scale=1.2] at (2,6.5) {$(u_2, w_4)$};

\node [scale=1.2] at (4,-0.4) {$(u_3, w_1)$};
\node [scale=1.2] at (4.8,2.4) {$(u_3, w_2)$};
\node [scale=1.2] at (4.8,4.4) {$(u_3, w_3)$};
\node [scale=1.2] at (4,6.5) {$(u_3, w_4)$};

\node [scale=1.2] at (6,-0.4) {$(u_4, w_1)$};
\node [scale=1.2] at (6.8,2.4) {$(u_4, w_2)$};
\node [scale=1.2] at (6.8,4.4) {$(u_4, w_3)$};
\node [scale=1.2] at (6,6.5) {$(u_4, w_4)$};

\node [scale=1.2] at (8,-0.4) {$(u_5, w_1)$};
\node [scale=1.2] at (8.8,2.4) {$(u_5, w_2)$};
\node [scale=1.2] at (8.8,4.4) {$(u_5, w_3)$};
\node [scale=1.2] at (8,6.5) {$(u_5, w_4)$};

\node [scale=1.2] at (10,-0.4) {$(u_6, w_1)$};
\node [scale=1.2] at (10.8,2.4) {$(u_6, w_2)$};
\node [scale=1.2] at (10.8,4.4) {$(u_6, w_3)$};
\node [scale=1.2] at (10,6.5) {$(u_6, w_4)$};

\node [scale=1.2] at (13,0) {$(u_7, w_1)$};
\node [scale=1.2] at (13,2) {$(u_7, w_2)$};
\node [scale=1.2] at (13,4) {$(u_7, w_3)$};
\node [scale=1.2] at (13,6) {$(u_7, w_4)$};

\end{tikzpicture}
\caption{Labeling of $P_7 \square P_4$.}\label{grid74}
\end{figure}
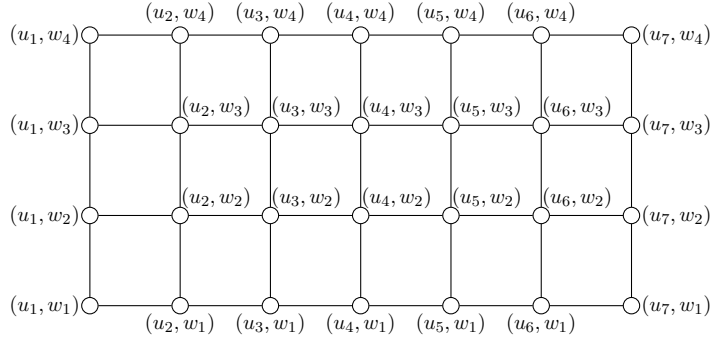

\begin{lemma}\label{grid_resolving}
For $s,t \ge 2$, let $G=P_s \square P_t$ be the grid graph with labeling as in Figure~\ref{grid74}. Let $S_1=\{(u_1,w_1), (u_1,w_t)\}$, $S_2=\{(u_1,w_1), (u_s,w_1)\}$, $S_3=\{(u_1,w_t), (u_s,w_t)\}$, and $S_4=\{(u_s,w_1), (u_s,w_t)\}$. Then $S$ is a minimum resolving set of $G$ if and only if $S=S_i$ for some $i \in \{1,2,3,4\}$.
\end{lemma}

\begin{proof} ($\Leftarrow$) We will show that $S_1$ forms a minimum resolving set of $G$. Let $v_1=(u_a, w_b)$ and $v_2=(u_c,w_d)$ be two distinct vertices in $G$. Note that $d((u_1, w_1), (u_a,w_b))=a-1+b-1$, $d((u_1, w_1), (u_c,w_d))=c-1+d-1$, $d((u_1, w_t), (u_a,w_b))=a-1+t-b$, and $d((u_1, w_t), (u_c,w_d))=c-1+t-d$. So, $(u_1, w_1)$ resolves $v_1$ and $v_2$ if $a+b \neq c+d$, and $(u_1, w_t)$ resolves $v_1$ and $v_2$ if $a-b \neq c-d$. If no vertex in $S_1$ resolves $v_1$ and $v_2$, then $a+b=c+d$ and $a-b=c-d$, which implies $a=c$ and $b=d$, contradicting the assumption that $v_1 \neq v_2$. Thus, $S_1$ is a minimum resolving set of $G$ with $|S_1|=2=\dim(G)$ by Proposition~\ref{dim_grid}. Similarly, one can easily show that $S_i$, $i \in \{2,3,4\}$, forms a minimum resolving set of $G$.\\

($\Rightarrow$) Let $\mathfrak{S}=\{S_1, S_2, S_3, S_4\}$. Let $R=\{(u_a,w_b), (u_x,w_y)\} \not\in \mathfrak{S}$ with $|R|=2$. We will show that $R$ is not a minimum resolving set of $G$.

\textbf{Case 1: $a=x$.} Without loss of generality, let $b<y$.

First, let $R \cap S_i \neq\emptyset$ for some $i \in \{1,2,3,4\}$, say $R \cap S_1 \neq \emptyset$, without loss of generality (other cases can be handled similarly). If $a=1$ and $b=1$, then $\code_{R}((u_1, w_{y+1}))=\code_{R}((u_2, w_y))$. If $x=1$ and $y=t$, then $\code_{R}((u_1, w_{b-1}))=\code_{R}((u_2, w_b))$. In each case, $R$ fails to be a resolving set of $G$.

Second, let $R \cap S_j=\emptyset$ for each $j \in \{1,2,3,4\}$. If $a=1$, then $\code_{R}((u_1,w_{b-1}))=\code_{R}((u_{2}, w_b))$ and $\code_{R}((u_1,w_{y+1}))=\code_{R}((u_2, w_y))$. The case $a=s$ is equivalent to the case $a=1$ by symmetry of the grid graph. If $a \not\in\{1,s\}$, then $\code_{R}((u_{a-1}, w_{b}))=\code_{R}((u_{a+1}, w_{b}))$ and $\code_{R}((u_{a-1}, w_{y}))=\code_{R}((u_{a+1}, w_{y}))$. In each case, $R$ fails to be a resolving set of $G$.

\textbf{Case 2: $b=y$.} Without loss of generality, let $a<x$.

First, let $R \cap S_i \neq\emptyset$ for some $i \in \{1,2,3,4\}$, say $R \cap S_1 \neq \emptyset$, without loss of generality (other cases are similar). If $a=1$ and $b=1$, then $\code_{R}((u_x, w_2))=\code_{R}((u_{x+1}, w_1))$. If $a=1$ and $b=t$, then $\code_{R}((u_x, w_{t-1}))=\code_{R}((u_{x+1}, w_t))$. In each case, $R$ fails to be a resolving set of $G$.

Second, let $R \cap S_j=\emptyset$ for each $j \in \{1,2,3,4\}$. If $b=1$, then $\code_{R}((u_{a-1}, w_1))=\code_{R}((u_a, w_2))$ and $\code_{R}((u_x, w_2))=\code_{R}((u_{x+1}, w_1))$. The case $b=t$ is equivalent to the case $b=1$ by symmetry of the grid graph. If $b \not\in \{1,t\}$, then $\code_{R}((u_a, w_{b-1}))=\code_{R}((u_a, w_{b+1}))$ and $\code_{R}((u_x, w_{b-1}))=\code_{R}((u_x, w_{b+1}))$. In each case, $R$ fails to be a resolving set of $G$.

\textbf{Case 3: $a \neq x$ and $b \neq y$.} Let $a<x$ and $b <y$ (other cases can be handled similarly). 

First, let $R \cap S_i \neq\emptyset$ for some $i \in \{1,2,3,4\}$, say $R \cap S_1\neq \emptyset$, without loss of generality (other cases are similar). Then $a=b=1$ and $\code_{R}((u_1, w_2))=\code_{R}((u_2, w_1))$, and thus $R$ fails to be a resolving set of $G$.

Second, let $R \cap S_j=\emptyset$ for each $j \in \{1,2,3,4\}$. If $a=1$ or $b=1$ (but not both), then $\code_{R}((u_a, w_{b+1}))=\code_{R}((u_{a+1}, w_b))$. If $x=s$ or $y=t$ (but not both), then $\code_{R}((u_{x-1}, w_y))=\code_{R}((u_x, w_{y-1}))$. If $a \neq 1$, $b \neq 1$, $x \neq s$, and $y \neq t$, then $\code_{R}((u_{a-1}, w_b))=\code_{R}((u_a, w_{b-1}))$ and $\code_{R}((u_x, w_{y+1}))=\code_{R}((u_{x+1}, w_y))$. In each case, $R$ fails to be a resolving set of $G$.~\hfill
\end{proof}

\begin{theorem}
For $s \ge t \ge 2$, let $G=P_s \square P_t$ be the grid graph with labeling as in Figure~\ref{grid74} and let $L=\cup_{i=1}^t\{(u_1,w_i),(u_s,w_i) \}$. Then
\begin{itemize}
\item[(a)] $\cdim(G)=t$;
\item[(b)] for $v \in V(G)$, $\cdim_G(v)=\left\{
\begin{array}{ll}
t & \mbox{ if } s=t \mbox{ and } \deg_G(v) \le 3, \mbox{ or } s>t \mbox{ and } v\in L,\\
t+1 & \mbox{ if } s=t \mbox{ and } \deg_G(v)=4, \mbox{ or } s>t \mbox{ and }v \not\in L.
\end{array}
\right.$
\end{itemize}
\end{theorem}

\begin{proof}
Let $X=V(G)-((\cup_{i=1}^{s}\{(u_i,w_1)\})\cup (\cup_{j=1}^{t}\{(u_1,w_j)\}))$. Then $X$ (and thus, any subset of $X$) fails to be a (connected) resolving set of $G$, since $\code_X((u_1,w_2))=\code_X((u_2,w_1))$.\\

(a) Note that $\min\{s,t\}=t$. Since $\cup_{i=1}^{t}\{(u_1, w_i)\}$ forms a connected resolving set of $G$ by Lemma~\ref{grid_resolving}, $\cdim(G) \le t$. Next, we show that $\cdim(G) \ge t$. Assume, to the contrary, that $\cdim(G) \le t-1$. If $S$ is a connected resolving set of $G$ with $|S| \le t-1$, we may assume that $S \subseteq X$, contradicting the assumption that $S$ is a resolving set of $G$. So, $\cdim(G) \ge t$, and thus $\cdim(G)=t$.\\  

(b) Let $L_1=\cup_{i=1}^{t}\{(u_1,w_i)\}$, $L_2=\cup_{i=1}^{t}\{(u_s,w_i)\}$, $L_3=\cup_{j=1}^{s}\{(u_j,w_1)\}$, and $L_4=\cup_{j=1}^{s}\{(u_j, w_t)\}$. Note that $L=L_1 \cup L_2$. 

\textbf{Case 1: $s=t$ and $\deg_G(v) \le 3$, or $s>t$ and $v\in L$.} Note that, for each $v\in V(G)$, $\deg_G(v) \in \{2,3,4\}$. First, let $s=t$ and $\deg_G(v) \le 3$. Then $v\in L_i$ for some $i \in \{1,2,3,4\}$, and $L_i$ is a connected resolving set of $G$ with $|L_i|=t=\cdim(G)$ by Lemma~\ref{grid_resolving} and (a) of the current theorem. Second, let $s>t$ and $v\in L$. Then $v\in L_j$ for some $j\in\{1,2\}$, and $L_j$ is a connected resolving set of $G$ with $|L_j|=t=\cdim(G)$ by Lemma~\ref{grid_resolving} and (a) of the current theorem. So, in each case, $\cdim_G(v) \le t$. Since $\cdim_G(v) \ge t$ by Observation~\ref{obs}(b) and (a) of the current theorem, we have $\cdim_G(v)=t$.

\textbf{Case 2: $s=t$ and $\deg_G(v)=4$, or $s>t$ and $v \not\in L$.} Let $v=(u_{\alpha},w_{\beta}) \in V(G)$. Let $D=\cup_{i=1}^t\{(u_{\alpha},w_i)\}$. Note that $v \in D$. First, we show that $D'=D \cup \{(u_{\alpha-1},w_{\beta})\}$ is a connected resolving set of $G$ with $|D'|=t+1$, and thus $\cdim_G(v) \le t+1$. Let $B_1=\cup_{i=1}^{t}\{(u_1,w_i),(u_2,w_i), \ldots, (u_{\alpha-1},w_i)\}$ and $B_2=\cup_{i=1}^{t}\{(u_{\alpha+1},w_i), (u_{\alpha+2},w_i), \ldots, (u_s,w_i)\}$. Let $x=(u_a,w_b)$ and $y=(u_c,w_d)$ be two distinct vertices in $G$. If $x,y \in B_1$ ($x,y \in B_2$, respectively), then $D$ is a connected resolving set of $G[B_1]$ ($G[B_2]$, respectively) by Lemma~\ref{grid_resolving}. So, suppose that $x\in B_1$ and $y \in B_2$, or $x\in B_2$ and $y \in B_1$, say the former. Note that $\code_D(x)=\code_D(y)$ implies that $b=d$ and $|\alpha-a|=|\alpha-c|$. Since $d((u_a,w_b),(u_{\alpha-1},u_{\beta}))=d((u_a,w_b),(u_{\alpha}, u_{\beta}))-1<d((u_a,w_b),(u_{\alpha}, u_{\beta}))+1=d((u_c,w_d),(u_{\alpha}, u_{\beta}))+1=d((u_c,w_d),(u_{\alpha-1}, w_{\beta}))$, we have $\code_{D'}(x) \neq \code_{D'}(y)$. So, $D'$ is a connected resolving set of $G$. Next, we show that $\cdim_G(v) \ge t+1$. Let $S$ be any connected resolving set of $G$ containing $v$. Since $\cdim_G(v) \ge \cdim(G)=t$ by Observation~\ref{obs}(b) and (a) of the current theorem, $|S| \ge t$. If $|S|=t$, we may assume that $S=D$ since no $t$ vertices in $X$ forms a resolving set of $G$. Since $\code_D((u_{\alpha-1},w_1))=\code_D((u_{\alpha+1},w_1))$, $D$ fails to be a resolving set of $G$. So, $|S| \ge |D|+1$, and hence $\cdim_G(v) \ge t+1$. Therefore, $\cdim_G(v)=t+1$.~\hfill
\end{proof}


\section{The effect of vertex or edge deletion on the connected metric dimension of graphs}

Throughout this section, let $v$ and $e$, respectively, denote a vertex and an edge of a connected graph $G$ such that both $G-v$ and $G-e$ are connected graphs. In~\cite{vedeletion}, it was shown that $\dim(G)-\dim(G-v)$ can be arbitrarily large (c.f.~\cite{wheel1}). It was also shown in~\cite{vedeletion} that $\dim(G-v)-\dim(G)$ can be arbitrarily large. We recall the following result on the effect of edge deletion on metric dimension of a graph.

\begin{theorem}\emph{\cite{vedeletion}}
\begin{itemize}
\item[(a)] For any graph $G$ and any edge $e \in E(G)$, $\dim(G-e) \le \dim(G)+2$.
\item[(b)] The value of $\dim(G)-\dim(G-e)$ can be arbitrarily large, as $G$ varies.
\end{itemize}
\end{theorem}

In particular, if $G-e$ is a tree (i.e., $G$ is a unicyclic graph), we have the following

\begin{theorem}\emph{\cite{tree1, comp1}}
Let $T$ be a tree of order at least three. If $e \in E(\overline{T})$, then
$$\dim(T)-2 \le \dim(T+e) \le \dim(T)+1.$$
\end{theorem}

Now, we examine the effect of vertex deletion on the connected metric dimension of a graph.

\begin{remark}
The value of $\cdim(G)-\cdim(G-v)$ can be arbitrarily large, as $G$ varies. Let $G=W_{n}$ for $n \ge 7$ and let $v$ be the central vertex of degree $n$ in $G$. Then $\cdim(G)=\lfloor \frac{2n+2}{5} \rfloor +1$ by Proposition~\ref{cdim_wheel}, and $\cdim(G-v)=\cdim(C_n)=2$ since any two adjacent vertices form a connected minimum resolving set of $C_n$. Thus, $\cdim(G)-\cdim(G-v)=\lfloor \frac{2n+2}{5} \rfloor -1$.
\end{remark}

\begin{remark}\label{cdim_vdel}
The value of $\cdim(G-v)-\cdim(G)$ can be arbitrarily large, as $G$ varies. For $k \ge 6$, let $G-v$ be a tree obtained from the $(k+2)$-path, given by $u_0,u_1, u_2, \ldots, u_k, u_{k+1}$, by joining a vertex $\ell_i$ to $u_i$ for each $i \in \{1,k\}$, and let $G$ be the graph obtained by joining $\ell_1$ and $\ell_k$ to a new vertex $v$ (see Figure~\ref{vdel}). Then $\cdim(G-v)=k+2$ by Theorem~\ref{cdim_tree}(a), and $\cdim(G) \le 7$ since $S=\{u_0, u_1, \ell_1, v, \ell_k, u_k, u_{k+1}\}$ forms a connected resolving set of $G$ with $|S|=7$. Thus, $\cdim(G-v)-\cdim(G) \ge k-5$.
\end{remark}

\begin{figure}[ht]
\centering
\begin{tikzpicture}[scale=.6, transform shape]

\node [draw, shape=circle] (0) at  (-2.3,1.3) {};
\node [draw, shape=circle] (1) at  (-1,1.3) {};
\node [draw, shape=circle] (2) at  (0.3,1.3) {};
\node [draw, shape=circle] (3) at  (1.6,1.3) {};
\node [draw, shape=circle] (4) at  (2.9,1.3) {};
\node [draw, shape=circle] (5) at  (5,1.3) {};
\node [draw, shape=circle] (6) at  (6.3,1.3) {};

\node [draw, shape=circle] (01) at  (-1,0) {};
\node [draw, shape=circle] (05) at  (5,0) {};

\node [draw, shape=circle] (v) at  (2,-1) {};

\draw(0)--(1)--(2)--(3)--(4)--(3.5,1.3); \draw[dotted, thick](3.5,1.3)--(4.4,1.3); \draw(4.4,1.3)--(5)--(6);
\draw(01)--(1);\draw(05)--(5); \draw(01)--(v)--(05);

\node [scale=1.2] at (-2.3,1.8) {$u_0$};\node [scale=1.2] at (-1,1.8) {$u_1$};\node [scale=1.2] at (0.3,1.8) {$u_2$};\node [scale=1.2] at (1.6,1.8) {$u_3$};\node [scale=1.2] at (2.9,1.8) {$u_4$};\node [scale=1.2] at (5.1,1.8) {$u_k$};\node [scale=1.2] at (6.3,1.8) {$u_{k+1}$};

\node [scale=1.2] at (-1,-0.5) {$\ell_1$};\node [scale=1.2] at (5,-0.5) {$\ell_k$};

\node [draw, shape=circle] (a0) at  (8.7,1.3) {};
\node [draw, shape=circle] (a1) at  (10,1.3) {};
\node [draw, shape=circle] (a2) at  (11.3,1.3) {};
\node [draw, shape=circle] (a3) at  (12.6,1.3) {};
\node [draw, shape=circle] (a4) at  (13.9,1.3) {};
\node [draw, shape=circle] (a5) at  (16,1.3) {};
\node [draw, shape=circle] (a6) at  (17.3,1.3) {};
\node [draw, shape=circle] (aL1) at  (10,0) {};
\node [draw, shape=circle] (aL5) at  (16,0) {};

\draw(a0)--(a1)--(a2)--(a3)--(a4)--(14.5,1.3); \draw[dotted, thick](14.5,1.3)--(15.4,1.3); \draw(15.4,1.3)--(a5)--(a6);\draw(a1)--(aL1);\draw(a5)--(aL5);

\node [scale=1.2] at (2,-2.5) {\textbf{(a)} $G$};
\node [scale=1.2] at (13,-2.5) {\textbf{(b)} $G-v$};\node [scale=1.2] at (2,-1.5) {$v$};

\node [scale=1.2] at (8.7,1.8) {$u_0$};\node [scale=1.2] at (10,1.8) {$u_1$};\node [scale=1.2] at (11.3,1.8) {$u_2$};\node [scale=1.2] at (12.6,1.8) {$u_3$};\node [scale=1.2] at (13.9,1.8) {$u_4$};\node [scale=1.2] at (16,1.8) {$u_k$};\node [scale=1.2] at (17.3,1.8) {$u_{k+1}$};

\node [scale=1.2] at (10,-0.5) {$\ell_1$};\node [scale=1.2] at (16,-0.5) {$\ell_k$};

\end{tikzpicture}
\caption{A graph $G$ such that $\cdim(G-v)-\cdim(G)$ can be arbitrarily large.}\label{vdel}
\end{figure}
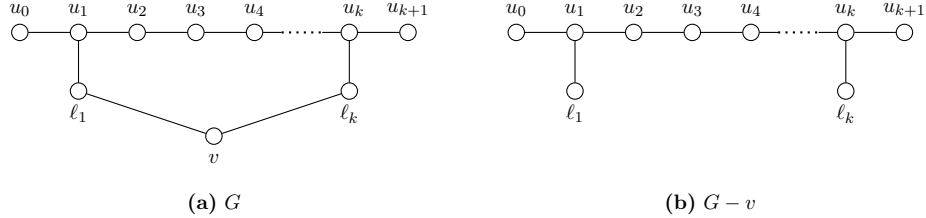

Next, we examine the effect of edge deletion on the connected metric dimension of a graph.

\begin{remark}
The value of $\cdim(G)-\cdim(G-e)$ can be arbitrarily large, as $G$ varies. Let $G-e$ be a tree with $ex(G-e)=1$. Let $w$ be the exterior major vertex of $G-e$, and let $\ell_1, \ell_2, \ldots, \ell_k$ be the terminal vertices of $w$ in $G-e$, where $k \ge 3$, such that $d_{G-e}(w, \ell_1) \ge d_{G-e}(w, \ell_2) \ge \ldots \ge d_{G-e}(w,\ell_k)=1$; further, let the $\ell_1-w$ path in $G-e$ be given by $\ell_1=s_0, s_1, s_2, \ldots, s_{a-1}, s_a=w$, with $d_{G-e}(w, \ell_1)=a \ge 4$. Let $e=\ell_1s_3$ be the new edge of $G$. See Figure~\ref{edel}. By Lemma~\ref{lemma_tree1}, $\cdim(G-e)=k$. Next, we show that $\cdim(G)=k+a-3$. Let $S$ be any minimum connected resolving set of $G$. We denote by $P^i$ the $w-\ell_i$ path, excluding $w$, where $2 \le i \le k$. Note that $\cdim(G) \ge k+a-3$: (i) $S \cap \{\ell_1, s_2\} \neq \emptyset$ since $\ell_1$ and $s_2$ are twin vertices in $G$; (ii) $S \cap V(P^i) \neq \emptyset$ for each $i \in \{2,3,\ldots, k\}$ with exactly one exception; (iii) $S$ must contain all vertices lying on the $\ell_1-w$ geodesic or $s_2-w$ geodesic in $G$ by connectedness of $G[S]$. Since $S'=(\cup_{i=2}^{a-2}\{s_i\})\cup(N[w]-\{\ell_k\})$ forms a connected resolving set of $G$ with $|S'|=a-3+k$, we have $\cdim(G)=k+a-3$. Thus, $\cdim(G)-\cdim(G-e)=k+a-3-k=a-3$.
\end{remark}

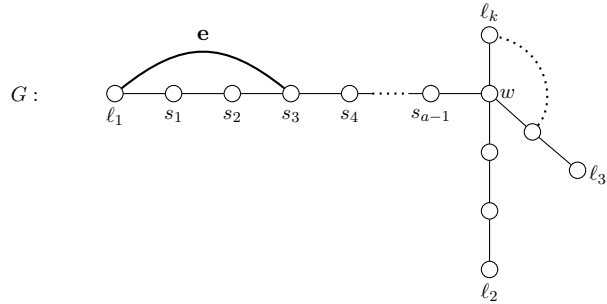
\begin{figure}[ht]
\centering
\begin{tikzpicture}[scale=.6, transform shape]

\node [draw, shape=circle] (0) at  (0,1.3) {};
\node [draw, shape=circle] (1) at  (1.3,1.3) {};
\node [draw, shape=circle] (2) at  (2.6,1.3) {};
\node [draw, shape=circle] (3) at  (3.9,1.3) {};
\node [draw, shape=circle] (4) at  (5.2,1.3) {};
\node [draw, shape=circle] (5) at  (7,1.3) {};
\node [draw, shape=circle] (w) at  (8.3,1.3) {};

\node [draw, shape=circle] (6) at  (8.3,0) {};
\node [draw, shape=circle] (7) at  (8.3,-1.3) {};
\node [draw, shape=circle] (8) at  (8.3,-2.6) {};

\node [draw, shape=circle] (9) at  (9.25,0.45) {};
\node [draw, shape=circle] (10) at  (10.25,-0.4) {};

\node [draw, shape=circle] (11) at  (8.3,2.6) {};

\draw(0)--(1)--(2)--(3)--(4)--(5.7,1.3); \draw[dotted, thick](5.7,1.3)--(6.5,1.3); \draw(6.5,1.3)--(5)--(w);
\draw(w)--(9)--(10);\draw(8)--(7)--(6)--(w)--(11);

\draw[dotted, thick](9) .. controls (9.8,1.1) and (9.6,2.3) .. (11);

\draw[thick](0) .. controls (1.5, 2.5) and (2.4,2.5) .. (3);
\node [scale=1.3] at (1.95,2.6) {${\bf{e}}$};

\node [scale=1.2] at (-2,1.3) {$\LARGE G:$};

\node [scale=1.2] at (0,0.8) {$\ell_1$};\node [scale=1.2] at (1.3,0.8) {$s_1$};\node [scale=1.2] at (2.6,0.8) {$s_2$};\node [scale=1.2] at (3.9,0.8) {$s_3$};\node [scale=1.2] at (5.2,0.8) {$s_4$};\node [scale=1.2] at (7,0.8) {$s_{a-1}$};\node [scale=1.2] at (8.7,1.3) {$w$};\node [scale=1.2] at (8.3,-3.1) {$\ell_2$};\node [scale=1.2] at (10.7,-0.5) {$\ell_3$};\node [scale=1.2] at (8.3,3.1) {$\ell_k$};

\end{tikzpicture}
\caption{A graph $G$ such that $\cdim(G)-\cdim(G-e)$ can be arbitrarily large.}\label{edel}
\end{figure}

\begin{remark}
The value of $\cdim(G-e)-\cdim(G)$ can be arbitrarily large, as $G$ varies. Let $G-e$ be a tree, as in Figure~\ref{vdel}(b), and let $e=\ell_1\ell_k$ be the new edge of $G$, where $k \ge 6$. Then $\cdim(G-e)=k+2$ by Theorem~\ref{cdim_tree}(a), and $\cdim(G) \le 6$ since $S=\{u_0, u_1, \ell_1,\ell_k, u_k, u_{k+1}\}$ forms a connected resolving set of $G$ with $|S|=6$. Thus, $\cdim(G-e)-\cdim(G) \ge k-4$.
\end{remark}


\section{Trees and unicyclic graphs $G$ satisfying $\cdim(G)=\dim(G)$}

As stated in Observation~\ref{obs}(a), $\cdim(G) \ge \dim(G)$ for any connected graph $G$. Moreover, it is shown in Remark~\ref{cdim_dim} that $\cdim(G)-\dim(G)$ can be arbitrarily large. So, it is natural to consider connected graphs $G$ for which $\cdim(G)=\dim(G)$ holds. Clearly, $\cdim(K_n)=\dim(K_n)$ and $\cdim(K_n-e)=\dim(K_n-e)$ for any $e \in E(K_n)$. However, it appears to be a challenging task to characterize all graphs $G$ satisfying $\cdim(G)=\dim(G)$. In this section, we characterize graphs $G$ satisfying $\cdim(G)=\dim(G)$ when $G$ is a tree or a unicyclic graph. First, we consider trees.

\begin{proposition}
Let $T$ be a tree of order $n \ge 2$. Then $\cdim(T)=\dim(T)$ if and only if $T=P_n$.
\end{proposition}

\begin{proof}
($\Leftarrow$) If $T=P_n$, $\cdim(P_n)=1=\dim(P_n)$ by Theorem~\ref{dim_characterization}(a) and Observation~\ref{obs}(a).

($\Rightarrow$) If $T$ is not a path, then $\dim(T)<\cdim(T)$ by Theorems~\ref{dim_tree} and~\ref{cdim_tree}(a).~\hfill
\end{proof}

Next, we consider unicyclic graphs. For $m\ge 3$, a unicyclic graph obtained from $C_m$ by adding at most one path to each vertex of $C_m$ is called a \emph{generalized $m$-sun} (see Figure~\ref{8sun} for an example of a generalized $8$-sun).

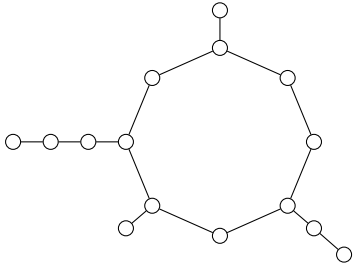
\begin{figure}[ht]
\centering
\begin{tikzpicture}[scale=.5, transform shape]

\node [draw, shape=circle, scale=1.1] (1) at  (0,2.5) {};
\node [draw, shape=circle, scale=1.1] (2) at  (-1.8,1.7) {};
\node [draw, shape=circle, scale=1.1] (3) at  (-2.5,0) {};
\node [draw, shape=circle, scale=1.1] (4) at  (-1.8,-1.7) {};
\node [draw, shape=circle, scale=1.1] (5) at  (0,-2.5) {};
\node [draw, shape=circle, scale=1.1] (6) at  (1.8,-1.7) {};
\node [draw, shape=circle, scale=1.1] (7) at  (2.5,0) {};
\node [draw, shape=circle, scale=1.1] (8) at  (1.8,1.7) {};

\node [draw, shape=circle, scale=1.1] (11) at  (0,3.5) {};
\node [draw, shape=circle, scale=1.1] (31) at  (-3.5,0) {};
\node [draw, shape=circle, scale=1.1] (32) at  (-4.5,0) {};
\node [draw, shape=circle, scale=1.1] (33) at  (-5.5,0) {};
\node [draw, shape=circle, scale=1.1] (41) at  (-2.5,-2.3) {};
\node [draw, shape=circle, scale=1.1] (61) at  (2.5,-2.3) {};
\node [draw, shape=circle, scale=1.1] (62) at  (3.3,-3) {};

\draw(1)--(2)--(3)--(4)--(5)--(6)--(7)--(8)--(1);
\draw(1)--(11); \draw(3)--(31)--(32)--(33); \draw(4)--(41);\draw(6)--(61)--(62);

\end{tikzpicture}
\caption{A generalized $8$-sun.}\label{8sun}
\end{figure}

\begin{lemma}\emph{\cite{vedeletion, comp1}}\label{lemma_joc}
Let $C$ be the unique cycle of a unicyclic graph $G$. Let $B_0=\{u,v, \theta\}$ and $B'_0=\{u_0, v_0, \theta_0\} \subseteq V(C)$, where $d(u_0,v_0)=\diam(C)$ and $u$ ($v$, $\theta$, respectively) is a vertex on the subtree rooted at $u_0$ ($v_0$, $\theta_0$, respectively). Then, we have $\code_{B_0}(x) \neq \code_{B_0}(y)$ for vertices $x$ and $y$ belonging to distinct subtrees rooted at vertices of the unique cycle $C$ of $G$. 
\end{lemma}

\begin{lemma}\label{lemma_unicyclic}
Let $G$ be a unicyclic graph with $\cdim(G)=\dim(G)$. Then $G$ is a generalized $m$-sun, where $m \ge 3$.
\end{lemma}

\begin{proof}
Let $T_u$ be the subtree rooted at a vertex $u$ which lies on the unique cycle $C$ of the graph $G$.

\textbf{Case 1: $T_u$ has a major vertex $q$, where $q \neq u$.} Let $S_c$ be a minimum connected resolving set of $G$. Let $q'$ be a neighbor of $q$ lying on the $u-q$ path, and let $P$ be the $u-q'$ path (if $q'=u$, let $P=\{u\}$). Let $T_q$ be the subtree of $T_u - V(P)$ that contains $q$. Since no vertex in $V(T_u)$ can resolve the two vertices in $N(u) \cap V(C)$, $S_c \cap (V(G)-V(T_u)) \neq \emptyset$; let $\theta \in S_c \cap (V(G)-V(T_u))$. Since no vertex in $V(G)-(V(T_q)-\{q\})$ can resolve any distinct vertices of $N(q) \cap V(T_q)$, $S_c \cap (V(T_q)-\{q\}) \neq \emptyset$; let $w \in S_c \cap (V(T_q)-\{q\})$. By the connectedness of $S_c$, $\{u,q,w,\theta\} \subseteq S_c$. Let $x$ and $y$ be two distinct vertices in $G$. We will show that $S=(S_c-\{u,q\}) \cup \{v\}$, where $v$ lies on $C$ and $d(u,v)=\diam(C)$, is a resolving set of $G$. First, let $x,y \in V(G)-V(T_u)$. Then $d(x,q) \neq d(y,q)$ or $d(x,u) \neq d(y,u)$ imply $d(x,w) \neq d(y,w)$; thus $\code_S(x) \neq \code_S(y)$. Second, let $x \in V(T_u)$ and $y \in V(G)-V(T_u)$. Then $\code_S(x) \neq \code_S(y)$ by Lemma~\ref{lemma_joc}. Third, let $x,y \in  V(T_u)$. Then $d(x,u) \neq d(y,u)$ implies $d(x,\theta) \neq d(y, \theta)$; thus $\code_{S\cup\{q\}}(x) \neq \code_{S \cup\{q\}}(y)$. This, together with the fact that a major vertex does not belong to any minimal resolving set of a tree (by the proof for Theorem~\ref{zhang_tree}), implies that $\code_S(x) \neq \code_S(y)$. Thus, we conclude $\cdim(G)>\dim(G)$ in this case.

\textbf{Case 2: $T_u$ has no major vertex, or the only major vertex of $T_u$ is $u$.} Suppose $T_u$ is not a path with $u$ as an end vertex. In this case, $u$ must belong to any connected resolving set of $G$. We will show that $u$ does not belong to any minimum resolving set of $G$. Since no vertex in the complement of $V(T_u)-\{u\}$ can resolve the neighbors of $u$ in $T_u$, any resolving set $S$ of $G$ must contain $v\in V(T_u)-\{u\}$; $S$ must also contain $w\in V(G)-V(T_u)$ to resolve neighbors of $u$ not in $V(T_u)$. Let $x$ and $y$ be two distinct vertices in $G$. We will show that if $v$ and $w$ cannot distinguish $x$ from $y$ by distance, then neither does $u$. 

First, let $x,y\in V(G)-V(T_u)$. Then $d(v,x)=d(v,u)+d(u,x)$ and $d(v,y)=d(v,u)+d(u,y)$. So, $d(v,x)=d(v,y)$ implies $d(u,x)=d(u,y)$. Second, let $x,y\in V(T_u)$. Then $d(w,x)=d(w,u)+d(u,x)$ and $d(w,y)=d(w,u)+d(u,y)$. So, again, $d(w,x)=d(w,y)$ implies $d(u,x)=d(u,y)$. Now, let $x\in V(T_u)$ and $y\in V(G)-V(T_u)$. If $x$ lies along the $v-u$ path, then $d(x,v)\neq d(y,v)$. If $u$ lies along the $x-v$ geodesic, then $d(x,v)=d(y,v)$ implies $d(x,u)=d(y,u)$. So, suppose $v$ lies along the $x-u$ path; also assume $d(x,v)=d(y,v)$ and $d(x,w)=d(y,w)$. Since $d(y,v)=d(y,u)+d(u,v)$ and $d(x,w)=d(x,v)+d(v,u)+d(u,w)$, we have $d(y,w)=d(y,u)+2d(v,u)+d(u,w)$, which is impossible since $d(v,u)\neq 0$.~\hfill 
\end{proof}

\begin{theorem}\emph{\cite{tree3}}\label{lemma_zhang}
Let $G$ be a generalized $m$-sun, where $m \ge 3$. Then $\dim(G)=2$ if $m$ is odd and $\dim(G) \le 3$ if $m$ is even.

\end{theorem}

Let $\mathcal{U}_1$ be the collection of all generalized ($2k+1$)-suns such that at least $2k-2$ consecutive vertices on the unique cycle are of degree two, where $k \ge 1$. Let $\mathcal{U}_2$ be the collection of all generalized $2k$-suns such that at least $2k-2$ consecutive vertices on the unique cycle are of degree two, where $k \ge 2$.

\begin{theorem}
Let $G$ be a unicyclic graph.
\begin{itemize}
\item[(a)] If the unique cycle of $G$ is an odd cycle, then $\cdim(G)=\dim(G)$ if and only if $G \in \mathcal{U}_1$.
\item[(b)] If the unique cycle of $G$ is an even cycle, then $\cdim(G)=\dim(G)$ if and only if $G \in \mathcal{U}_2$.
\end{itemize}
\end{theorem}

\begin{proof}
(a) Let $\mathcal{C}$ be the unique cycle of $G$ given by $u_1, u_2, \ldots, u_{2k+1}, u_1$, where $k \ge 1$.

($\Leftarrow$) Let $G \in \mathcal{U}_1$. Without loss of generality, let $\deg_G(u_i)=2$ for each $i \in \{1,2,\ldots, k-1\} \cup \{k+3, k+4, \ldots, 2k+1\}$. One can readily check that $S=\{u_1, u_{2k+1}\}$ forms a connected resolving set of $G$; thus $\cdim(G) \le 2$. By Observation~\ref{obs}(a) and Theorem~\ref{lemma_zhang}, $\cdim(G)=2=\dim(G)$.

($\Rightarrow$) Let $\cdim(G)=\dim(G)$. By Lemma~\ref{lemma_unicyclic}, $G$ must be a generalized ($2k+1$)-sun, where $k \ge 1$. If $\deg_G(u_i)=3$, let $\ell_i$ be the terminal vertex of $u_i$ and let $s_i$ be the vertex adjacent to $u_i$ lying on the $u_i-\ell_i$ path, where $i \in \{1,2,\ldots, 2k+1\}$. By Theorem~\ref{lemma_zhang}, $\dim(G)=2$. Note that $\{u_i, s_i\}$ does not distinguish the two neighbors of $u_i$ on $\mathcal{C}$ by distance. So, $\cdim(G)=\dim(G)$ implies that two adjacent vertices of $\mathcal{C}$ form a minimum resolving set of $G$; without loss of generality, let $S=\{u_1, u_{2k+1}\}$ be a resolving set of $G$. Then $\code_S(s_i)=\code_S(u_{i+1})$ for each $i \in \{1,2,\ldots, k-1\}$, and $\code_S(u_j)=\code_S(s_{j+1})$ for each $j \in \{k+2,k+3,\ldots, 2k\}$; thus, $\deg_G(u_i)=2$ for each $i \in \{1,2,\ldots, k-1\} \cup  \{k+3,k+4,\ldots, 2k+1\}$. Therefore, $G \in \mathcal{U}_1$.\\

(b) Let $\mathcal{C}$ be the unique cycle of $G$ given by $u_1, u_2, \ldots, u_{2k}, u_1$, where $k \ge 2$.

($\Leftarrow$) Let $G \in \mathcal{U}_2$. Without loss of generality, let $\deg_G(u_i)=2$ for each $i \in \{1,2,\ldots, k-1\} \cup \{k+2, k+3, \ldots, 2k\}$. One easily checks that $S=\{u_1, u_{2k}\}$ forms a connected resolving set of $G$; thus, $\dim(G) \le \cdim(G) \le 2$ by Observation~\ref{obs}(a). By Theorem~\ref{dim_characterization}(a), $\dim(G)=2=\cdim(G)$.

($\Rightarrow$) Let $\cdim(G)=\dim(G)$. By Lemma~\ref{lemma_unicyclic}, $G$ must be a generalized $2k$-sun, where $k \ge 2$. So, $\dim(G) \in \{2,3\}$ by Theorems~\ref{dim_characterization}(a) and~\ref{lemma_zhang}. For $i \in \{1,2,\ldots, 2k\}$, let $\ell_i$ be the terminal vertex of $u_i$ if $\deg_{G}(u_i)=3$, and let $u_i=\ell_i$ if $\deg_G(u_i)=2$; further, if $\deg_G(u_i)=3$, let $s_i$ be the vertex adjacent to $u_i$ lying on the $u_i-\ell_i$ path.

\textbf{Case 1: $\cdim(G)=\dim(G)=2$.} We will show that $G \in \mathcal{U}_2$ in this case. Without loss of generality, let $S=\{u_1, u_{2k}\}$ be a minimum (connected) resolving set of $G$. Then $\code_S(s_i)=\code_S(u_{i+1})$ for each $i \in \{1,2,\ldots, k-1\}$, and $\code_S(u_j)=\code_S(s_{j+1})$ for each $j \in \{k+1,k+2,\ldots, 2k-1\}$; thus, $\deg_G(u_i)=2$ for each $i \in \{1,2,\ldots, k-1\} \cup  \{k+2,k+3,\ldots, 2k\}$. Therefore, $G \in \mathcal{U}_2$.

\textbf{Case 2: $\cdim(G)=\dim(G)=3$.} We will show that there is no such unicyclic graph $G$. If $k=2$, then $\dim(G)=2$ for any generalized 4-sun since $\{\ell_1, \ell_2\}$ forms a minimum resolving set of $G$. Thus, assume $k \ge 3$. Let $S$ be any minimum connected resolving set of $G$ with $|S|=3$; then, clearly, $|S \cap V(\mathcal{C})| \ge 2$.

First, let $|S\cap V(\mathcal{C})|=3$. Without loss of generality, let $S=\{u_{2k},u_1,u_2\}$ be a connected resolving set of $G$. Then $\code_{S}(s_i)=\code_{S}(u_{i+1})$ for each $i \in \{2,\ldots, k-1\}$, and $\code_{S}(u_j)=\code_{S}(s_{j+1})$ for each $j \in \{k+2,\ldots, 2k-1\}$; thus, $\deg_G(u_i)=2$ for each $i \in \{2,\ldots, k-1\} \cup  \{k+3,\ldots, 2k\}$. If we let $G'$ be a generalized $2k$-sun, for $k \ge 3$, with the unique cycle $\mathcal{C}$ such that $\deg_{G'}(u_i)=2$ for each $i \in \{2,3,\ldots, k-1\} \cup  \{k+3,k+4,\ldots, 2k\}$, then $\dim(G') \le 2$ since $\{\ell_1, \ell_k\}$ forms a resolving set of $G'$. So, there is no unicyclic graph $G$ satisfying $\cdim(G)=\dim(G)=3$ with $|S\cap V(\mathcal{C})|=3$.

Second, let $|S\cap V(\mathcal{C})|=2$. By a relabeling of the vertices if necessary, let $\deg_G(u_1)=3$ and let $S=\{u_1, s_1,u_{2k}\}$ be a connected resolving set of $G$. Then $\code_{S}(s_i)=\code_{S}(u_{i+1})$ for each $i \in \{2,\ldots, k-1\}$, and $\code_{S}(u_j)=\code_{S}(s_{j+1})$ for each $j \in \{k+1,\ldots, 2k-1\}$; thus, $\deg_G(u_i)=2$ for each $i \in \{2,\ldots, k-1\} \cup  \{k+2,\ldots, 2k\}$. If we let $G^*$ be a generalized $2k$-sun, for $k \ge 3$, with the unique cycle $\mathcal{C}$ such that $\deg_{G^*}(u_i)=2$ for each $i \in \{2,3,\ldots, k-1\} \cup  \{k+2,k+3,\ldots, 2k\}$, then $\dim(G^*) \le 2$ since $\{\ell_1, \ell_k\}$ forms a resolving set of $G^*$. So, there is no unicyclic graph $G$ satisfying $\cdim(G)=\dim(G)=3$ with $|S\cap V(\mathcal{C})|=2$.~\hfill
\end{proof}


\section{Open problems}

We conclude this paper with some open problems.\\

1. As stated in Observation~\ref{obs_transitive}, if $G $ is vertex-transitive, then $\cdim_G(v)=\cdim(G)$ for any $v \in V(G)$. On the other hand, an example of a graph $H$ that is not vertex-transitive and $\cdim_H(v)=\cdim(H)$, for any $v \in V(H)$, was provided in Remark~\ref{non_transitive}. Can we characterize all connected graphs $G$ such that $\cdim_G(v)=\cdim(G)$ for any $v\in V(G)$?\\

2. In Proposition~\ref{rrad_rdiam}, it was shown that $\rrad(G) \le \rdiam(G) \le \rrad(G)+\diam(G)$. Can we characterize graphs $G$ satisfying $\rdiam(G)=\rrad(G)$ or $\rdiam(G)=\rrad(G)+\diam(G)$?\\

3. In section~7, we characterized trees and unicyclic graphs such that connected metric dimension equals metric dimension. Can we characterize other graph classes where $\cdim(G)=\dim(G)$ for a member $G$ of the graph class?\\

4. For $S \subseteq V(G)$, we define the \emph{connected metric dimension of $G$ at $S$}, denoted by $\cdim_G(S)$, to be the minimum order of a resolving set of $G$, containing $S$, that induces a connected subgraph of $G$. In this paper, we considered this problem when $S$ is a singleton. It seems interesting to study $\cdim_G(S)$ for $|S| > 1$.\\

5. Let $B_m$ be a bouquet of $m$ cycles $C^1, C^2, \ldots, C^m$ with the cut-vertex $w$, where $m \ge 2$. Let $C^1$ be an odd cycle given by $w, u_1, u_2, \ldots, u_{2k}, w$, where $k \ge 3$, and let $C^m$ be an even cycle. Let $S$ be any minimum resolving set of $B_m$, and let $S_c$ be any minimum connected resolving set of $B_m$. We note that $|S \cap \{u_k,u_{k+1}\}|=1$ (see~\cite{bouquet}), whereas $S_c \cap \{u_k, u_{k+1}\}=\emptyset$. Likewise, any minimum connected resolving set of $B_m$ at $u_{i} \in V(C^1)-(N(u_k) \cup N(u_{k+1}))$ contains neither $u_k$ nor $u_{k+1}$. This shows that, for a general graph $G$, the problem of determining $\dim(G)$ is essentially different from that of $\cdim(G)$. So, it is a natural problem to consider the computational complexity of determining the connected metric dimension (at a vertex) of $G$.\\\\

\textbf{Acknowledgments.} The authors are grateful to the anonymous referees for the corrections, as well as the valuable comments, which substantially improved this paper.


\end{document}